\documentclass[oneside]{amsart}
\usepackage[utf8]{inputenc}
\usepackage{graphicx}
\usepackage{amsmath,amssymb}
\numberwithin{figure}{subsection}
\usepackage{setspace}
\usepackage{float}
\usepackage{geometry}
\usepackage{caption}
\usepackage{amssymb}
\usepackage{listings}
\usepackage{amsthm}
\usepackage{graphicx}
\usepackage{biblatex}
\usepackage{amsthm}
\usepackage{hyperref}
\usepackage{lineno}
\usepackage{mathrsfs}
\usepackage{color}
\usepackage{tikz}
\newtheorem{theorem}{Theorem}[subsection]
\newtheorem{proposition}[theorem]{Proposition}
\newtheorem{corollary}[theorem]{Corollary}

\newtheorem{lemma}[theorem]{Lemma}
\theoremstyle{definition}
\newtheorem{definition}[theorem]{Definition}

\newtheorem{remark}[theorem]{Remark}

\usepackage{amsthm}

\title[Hausdorff Dimension and Lebesgue Measure of Codiagonal]{Hausdorff Dimension and Lebesgue Measure of Codiagonal of Embedded Vector Bundles over Submanifolds \\ in EUCLIDEAN Space}
\author{Hanwen Liu}

\begin{document}
\maketitle
\begin{abstract}
    In this paper we study measure theoretical size of the image of naturally embedded vector bundles in $\mathbb{R}^{n} \times \mathbb{R}^{n}$ under the codiagonal morphism, i.e. $\Delta_{*}$ in the category of finite dimensional $\mathbb{R}$-vector spaces. Under very weak smoothness condition we show that codiagonal of normal bundles always contain an open subset of the ambient space, and we give corresponding criterions for the tangent bundles. For any differentiable hypersurface we show that the codiagonal of its tangent bundle has non-empty interior, unless the hypersurface is contained in a hyperplane. Assuming further smoothness (e.g. twice differentiable) we show that union of any family of hyperplanes that covers the hypersurface has maximal possible Hausdorff dimension. We also define and study a notion of degeneracy of embedded $C^{1}$ vector bundles over a $C^{1}$ submanifold and show as a corollary that if the base manifold has at least one non-inflection point then codiagonal of any $C^{1}$ line bundle over it has positive Lebesgue measure. Finally we show that codiagonal of any line bundle over an $n$-dimensional ellipsoid or a convex curve has non-empty interior, and the same assertion also holds for any non-tangent line bundle over a hyperplane.
\end{abstract}

\section{Introduction and preparation lemmas}
\label{sec1}
In this paper all base fields are assumed to be the field of real numbers $\mathbb{R}$, and all vector spaces $\mathbb{R}^{n}:=\left(\mathbb{R}^{n},\langle\cdot, \cdot\rangle\right)$ are assumed to be equipped with the standard Euclidean structure.

For any subset $X \subseteq \mathbb{R}^{n}$ we denote by $i_{X}: X \longrightarrow \mathbb{R}^{n}$ the canonical inclusion and by $\operatorname{int} X, \bar{X}, \partial X$ its interior, closure, boundary (in the induced topology) resp. We denote $\mathbb{R}_{+}:=\{x \in \mathbb{R} \mid x>0\}$, and $I$ will always be the open interval $]-\pi, \pi[$. As usual, $\mathbb{D}^{n}:=\left\{x \in \mathbb{R}^{n} \mid\|x\| \leqslant 1\right\}$ is the $n$-dimensional unit disk, $\mathbb{S}^{n-1}:=\left\{x \in \mathbb{R}^{n} \mid\|x\|=1\right\}$ is the $(n-1)$ - dimensional unit sphere, and $B_{r}^{n}(a):=\left\{x \in \mathbb{R}^{n} \mid\|x-a\|<r\right\}$ is the $n$-ball with center $a \in \mathbb{R}^{n}$ and radius $r$. For simplicity we write $B_{r}^{n}:=B_{r}^{n}(0)$.

For any map $f$ denote by $\Gamma_{f}$ its graph, and for any differentiable $f$ denote by $J_{x}f$ its Jacobian at $x$. Linear functions and linear mappings on an open subset of an Euclidean space are defined to be the 
functions and mappings with constant differential.

Moreover we denote by $\operatorname{dim}$ and  $\operatorname{dim}_{\mathcal{H}}$ the topological and Hausdorff dimension resp. and by $\mathcal{L}^{n}$ the $n$-dimensional Lebesgue measure.

\subsection{Manifolds in Euclidean Space}
In this paper, a topological $d$-submanifold (with boundary) $M$ of $\mathbb{R}^{n}$ is a topological subspace of $\mathbb{R}^{n}$ s.t. $M$ is a topological $d$-manifold (with boundary) in the induced topology. Moreover, topological $d$-submanifold $M$ is said to be a $C^{r}$ submanifold iff its transition functions are of class $C^{r}$. All manifolds are assumed to be connected. 

In this paper we also consider the following two notions of smoothness.
\begin{definition}\label{defin1_1}
    A topological submanifold $M$ is said to be $k$-th differentiable iff for any $p \in M$, at the vicinity of $p, M$ coincides with the graph of a $k$-th differentiable map $f$. For simplicity once differentiable submanifolds are abbreviated as differentiable submanifolds.
\end{definition}
\begin{definition}\label{defin1_2}
    A $C^{1}$ submanifold is said to be Lipschitz-continuously differentiable iff for any $p \in M$, at the vicinity of $p, M$ coincides with the graph of a $C^{1}$ map $f$ with locally Lipschitz differential.
\end{definition}

Since $C^{1}$ maps are locally Lipschitz, by definition we have that any $C^{2}$ submanifold is Lipschitz-continuously differentiable.

As usual, a 1-submanifold is called a regular curve, a 2-submanifod is called a regular surface, a submanifold of codimension 1 is called a hypersurface.

Let $M$ be a $k$-submanifold of $\mathbb{R}^{m}$ and $N$ be a $\ell$-submanifold of $\mathbb{R}^{n}$ then the product manifold $M \times N:=\left\{(p, q) \in \mathbb{R}^{n+m} \mid p \in M, q \in N\right\}$ is a $(k+\ell)$-submanifold of $\mathbb{R}^{n+m}$. For any two maps $f: A \rightarrow C$ and $g: B \rightarrow D$ we define $f \times g: A \times B \rightarrow C \times D$ by $f \times g(a, b):=(f(a), g(b))$, then $\Gamma_{f \times g}=\Gamma_{f} \times \Gamma_g$. For example, if $f, g$ are differentiable maps (from open subets in Euclidean spaces to Euclidean spaces), then $\Gamma_f, \Gamma_g$ are differentiable submanifolds and $\Gamma_{f \times g}$ is their product submanifold.

\begin{remark}\label{remark1_3}
    Notice that the implicit and inverse function theorem hold for differentiable maps (Theorem 2 and 3 in \cite{6}), and differentiable maps with differentiable local inverses form a pseudo group $\mathcal{G}$ and hence defines a manifold structure. The differentiable submanifold in Definition \ref{defin1_1} clearly has local parametrization $\mathbf{f}(x)=(x, f(x))$ and admits a differentiable chart, and hence a $\mathcal{G}$-structure.
    
    The similar assertion holds also for Lipschitz-continuously differentiable submanifolds in Definition \ref{defin1_2}.
\end{remark}

\subsection{Substantial and Nonlinear Submanifolds}
Recall that a submanifold is said to be substantial iff it is not contained in any hyperplane of the ambient space.

An affine plane in $\mathbb{R}^{3}$ is not substantial, and the twisted quartic $\mathscr{X}: t \mapsto\left(t, t^{2}, t^{3}, t^{4}\right)$ is a well-known example of a substantial $C^{\omega}$ curve in $\mathbb{R}^{4}$.

Notice that the curvature of $\mathscr{X}$ is non-vanishing. Another nontrivial example can be constructed as follows.

Recall that for any $C^{1}$ regular curve $g: I \rightarrow \mathbb{R}^{n}$ the tangent derelopable $\Sigma$ of $g$ is parametrized by $\mathbf{x}(u, v)=g(u)+v g^{\prime}(u), u \in I, v \in \mathbb{R}$, and $g$ is a cuspidal edge of $\Sigma$.

We denote by $\Sigma_{+}$the one-sided tangent developable of $g$ parametrized by $\mathbf{x}(u, v)=g(u)+v g^{\prime}(u), u \in I, v \in \mathbb{R}_{+}$.

\begin{proposition}\label{prop1_4}
    Let $r \geqslant 1$ and $g: I \rightarrow \mathbb{R}^{n}$ be a substantial $C^{r+1}$ regular curve with non-vanishing curvature, then the one-sided tangent developable of $\Sigma_{+}$ of $g$ is a substantial $C^{r}$ regular surface.
\end{proposition}
\begin{proof}
    Since $\dfrac{\partial \mathbf{x}}{\partial u}=g^{\prime}(u)+v g^{\prime \prime}(u)$ and $\dfrac{\partial \mathbf{x}}{\partial v}=g^{\prime}(u)$, we have $\left\|\dfrac{\partial \mathbf{x}}{\partial u} \times \dfrac{\partial \mathbf{x}}{\partial v}\right\|=\left\|\left(g^{\prime}(u)+v g^{\prime \prime}(u)\right) \times g^{\prime}(u)\right\|$ $=v\left\|g^{\prime \prime}(u) \times g^{\prime}(u)\right\|$. Since $v>0$, we have that $\left\|\dfrac{\partial \mathbf{x}}{\partial u} \times \dfrac{\partial \mathbf{x}}{\partial v}\right\|>0$ iff $\left\|g^{\prime \prime}(u) \times g^{\prime}(u)\right\|>0$. Since the curvature $\kappa$ of $g$ is $\kappa=\dfrac{\left\|g^{\prime \prime}(u) \times g^{\prime}(u)\right\|}{\left\|g^{\prime}(u)\right\|^{3}}>0$, we obtain that $\left\|g^{\prime \prime}(u) \times g^{\prime}(u)\right\|>0$. Therefore $\left\|\dfrac{\partial \mathbf{x}}{\partial u} \times \dfrac{\partial \mathbf{x}}{\partial v}\right\| \neq 0$ and hence $\Sigma_{+}$ is a regular surface.
    
    Moreover $\mathbf{x}(u, v)=g(u)+v g^{\prime}(u)$ is clearly $C^{r}$. It remains to prove that $\Sigma_{+}$ is substantial.

    Provided that $\Sigma_+$ is not substantial then $\exists$ hyperplane $\Pi \subseteq \mathbb{R}^{4}$ s.t. $\Sigma_{+} \subseteq \Pi$. Since $\forall\ u \in I \lim \limits_{v \rightarrow 0^{+}} \mathbf{x}(u, v)=g(u)$, we have that $\forall\ p \in g, p$ is a limit point of $\Sigma_{+}$. Since $\Pi$ is closed, we have that $g \subseteq \Pi$, and hence $g$ is not substantial, contradiction.
\end{proof}

We give a definition of a notion related to the substantialness.
\begin{definition}\label{defin1_5}
    A submanifold $M$ of $\mathbb{R}^{n}$ is said to be linear iff $\exists$ affine subspace $\Pi \subseteq \mathbb{R}^{n}$ s.t. $M$ is an open subset of $\Pi$. Otherwise $M$ is said to be nonlinear.
\end{definition}

It is clear by definition that a linear substantial submanifold is an open subset in the ambient space, and for hypersurfaces the nonlinearity is equivalent to substantialness.

In this paper the nonlinearity is an important notion. We shall introduce several preparation lemmas about the nonlinearity.

\subsection{Restriction of the Domain of Nonlinearity}
Firstly we show that the nonlinearity of a function is a local property. 
\begin{lemma}\label{lemma1_6}
    Let $\Omega$ be a connected open subset of $\mathbb{R}^{d}$ and $f: \Omega \rightarrow \mathbb{R}$ be a real function. If for any $x \in \Omega\ (f, x)$ is the germ of a linear function, then $f$ is linear.
\end{lemma}
\begin{proof}
    Since $\forall\ x \in \Omega\ (f, x)$ is a germ of linear (and hence harmonic) function, we have that $f$ is harmonic on $\Omega$. Fix $x_{0} \in \Omega$, by assumption $\exists$ linear function $g: \Omega \rightarrow \mathbb{R}$ s.t. $\left(f, x_{0}\right)=\left(g, x_{0}\right)$. Since $\Omega$ is a connected open subset of $\mathbb{R}^{d}$, by the analytic continuation principle we have $f \equiv g$ and hence $f$ is linear.
\end{proof}

\begin{corollary}\label{coro1_7}
    Let $\Omega$ be a connected open subset of $\mathbb{R}^{d}$ and $f: \Omega \rightarrow \mathbb{R}$ is a nonlinear function, then $\exists\ x_{0} \in \Omega$ s.t. $\left(f, x_{0}\right)$ is the germ of a nonlinear function.
\end{corollary}
\begin{corollary}\label{coro1_8}
    Let $\Omega$ be a connected open subset of $\mathbb{R}^{d}$ and $f: \Omega \rightarrow \mathbb{R}^{m}$ is a nonlinear map, then $\exists\ x_{0} \in \Omega$ s.t. $\forall$ open neighborhood $U$ of $x_{0},\left.f\right|_{U}$ is nonlinear.
\end{corollary}
\begin{proof}
    Recall that a map $f=\left(f^{1}, \cdots, f^{m}\right)$ is nonlinear iff $\exists\ i \in\{1, \cdots, m\}$ s.t. $f^{i}$ is a nonlinear function. Apply Corollary \ref{coro1_7}.
\end{proof}
\begin{remark}\label{remark1_9}
    From Corollary \ref{coro1_8} we obtain that for any nonlinear map $f$ defined on $\Omega$, up to rigid motion and homothety we may assume that $I^{d} \subseteq \Omega$ and $f$ is nonlinear on $I^{d}$.
\end{remark}    

As an application we show that the nonlinearity of a submanifold is a local property. 
\begin{lemma}\label{lemma1_10}
    If a differentiable submanifold $M$ of $\mathbb{R}^{n}$ is locally graph of linear maps  at all $p \in M$ then $M$ is a linear submanifold.
\end{lemma}
\begin{proof}
    By assumption clearly $M$ admits an analytic atlas and hence is a $C^{\omega}$ submanifold. Fix $p \in M, \exists$ local chart $U$ of $M$ at the vicinity of $p$ and an affine subspace $\Pi \subseteq \mathbb{R}^{n}$ s.t. $M \cap \Pi \supseteq U$. Since $U$ is open in $M$ and $\Pi$, and $\Pi$ is complete, by the uniquess theorem of $C^{\omega}$ submanifolds, $M$ is an open subset of $\Pi$, i.e. a linear submanifold.
\end{proof}

\begin{corollary}\label{coro1_11}
    Let $M$ be a nonlinear differentiable $d$-submanifold of $\mathbb{R}^{n}$, then up to rigid motion and homothety, $\exists$ local chart $U$ of $M$ s.t. $U=\Gamma_{f}$ where $f: I^{d} \rightarrow \mathbb{R}^{n-d}$ is a nonlinear differentiable map.
\end{corollary}
\begin{proof}
    By Lemma \ref{lemma1_10} $\exists$ local chart $U$ of $M$ s.t. $U=\Gamma_{f}$ where $f: \Omega \rightarrow \mathbb{R}^{n-d}$ is a nonlinear differentiable map. By Remark \ref{remark1_9} wlog we assume $\Omega=I^{d}$.
\end{proof}

We now introduce another type of restriction of domain of nonlinearity. We start with several lemmas.

Recall that a monomial $x^{\alpha} \in \mathbb{R}\left[x_{1}, \cdots, x_{n}\right]$ is said to be square free iff $\|\alpha\|_{\infty} \leqslant 1$, where $\alpha \in \mathbb{N}^{n}$ is the multi-index. For simplicity we denote $\Xi_{n}:=\operatorname{span}_{\mathbb{R}}\left\{x^{\alpha} \in \mathbb{R}\left[x_{1}, \cdots, x_{n}\right] \mid\|\alpha\|_{\infty} \leqslant 1\right\}$ as a space of polynomial functions defined on $I^{n}$.

\begin{lemma}\label{lemma1_12}
    Let $f: I^{n} \rightarrow \mathbb{R}$ be a real function s.t. $\forall\ i=1, \cdots, n \exists$ real function $a_{i}, b_{i}: I^{n-1} \rightarrow \mathbb{R}$ s.t. $f\left(x_{1}, \cdots, x_{n}\right)=a_{i}\left(x_{1}, \cdots, \hat{x}_{i}, \cdots, x_{n}\right) x_{i}+b_{i}\left(x_{1}, \cdots, \hat{x}_{i}, \cdots, x_{n}\right)$, then $f \in \Xi_{n}$.
\end{lemma}
\begin{proof}
    Apply induction on $n$.
    
    The statement holds trivially for $n=1$.
    
    Assume that the statement holds for $n=k$. Take any $f: I^{k+1} \rightarrow \mathbb{R}$ s.t. $\forall\ i=1, \cdots, k+1 \exists\ a_{i}, b_{i}: I^{k} \rightarrow \mathbb{R} \text { s.t. } f\left(x_{1}, \cdots, x_{k+1}\right)=a_{i}\left(x_{1}, \cdots, \hat{x}_{i}, \cdots, x_{k+1}\right) x_{i}+b_{i}\left(x_{1}, \cdots, \hat{x}_{i}, \cdots, x_{k+1}\right)$. For any $t \in I, \forall\ i=1, \cdots, k$ define $a_{i, t}, b_{i, t}: I^{k-1} \rightarrow \mathbb{R}$ by $a_{i, t}\left(y_{1}, \cdots, y_{k-1}\right):=a_{i}\left(y_{1}, \cdots, y_{k-1}, t\right), b_{i, t}\left(y_{1}, \cdots,\right.$ $\left. y_{k-1}\right):=b_{i}\left(y_{1}, \cdots,y_{k-1}, t\right)$, then $f_{t}\left(x_{1}, \cdots, x_{k}\right):=f\left(x_{1}, \cdots, x_{k}, t\right)=a_{i}\left(x_{1}, \cdots, \hat{x}_{i}, \cdots, x_{k}, t\right) x_{i}+b_{i}\left(x_{1} \cdots, \hat{x}_{i}, \cdots, x_{k}, t\right)=a_{i, t}\left(x_{1}, \cdots, \hat{x}_{i}, \cdots, x_{k}\right) x_{i}+b_{i, t}\left(x_{1}, \cdots, \hat{x}_{i}, \cdots, x_{k}\right)$. By the induction hypothesis $f_{t}\in \Xi_{k}$. Since $b_{k+1}\left(x_{1}, \cdots, x_{k}\right)=a_{k+1}\left(x_{1}, \cdots, x_{k}\right) \cdot 0+b_{k+1}\left(x_{1}, \cdots, x_{k}\right)=f\left(x_{1}, \cdots, x_{k}, 0\right)$ $=f_{0}\left(x_{1},\cdots, x_{k}\right), a_{k+1}\left(x_{1}, \cdots, x_{k}\right)+b_{k+1}\left(x_{1}, \cdots, x_{k}\right)=a_{k+1}\left(x_{1}, \cdots, x_{k}\right) \cdot 1+b_{k+1}\left(x_{1}, \cdots, x_{k}\right)=f\left(x_{1}, \cdots, x_{k}, 1\right)=f_{1}\left(x_{1}, \cdots, x_{k}\right)$, we have $b_{k+1}=f_{0} \in \Xi_{k}$ and $a_{k+1}=\left(a_{k+1}+b_{k+1}\right)-b_{k+1}=f_{1}-f_{0} \in \Xi_{k}$. Since $f\left(x_{1}, \cdots, x_{k}\right)=a_{k+1}\left(x_{1}, \cdots, x_{k-1}\right) x_{k}+b_{k+1}\left(x_{1}, \cdots, x_{k-1}\right)$, it is clear that $f=a_{k+1} x_{k+1}+b_{k+1} \in \Xi_{k+1}$. Therefore the statement holds for $n=k+1$.
    
    This concludes the proof.
\end{proof}

\begin{lemma}\label{lemma1_13}
    Let $f: I^{n} \rightarrow \mathbb{R}$ be a real function, if for any isometric embedding $\ell: I \rightarrow I^{n}$ we have $f \circ l$ is a linear function, then $f$ is a linear function.
\end{lemma}
\begin{proof}
    For any $i=1, \cdots, n, \forall\ a_{1}, \cdots, \hat{a}_{i}, \cdots, a_{n} \in I$, denote $\alpha:=\left(a_{1}, \cdots, a_{i-1},0, a_{i+1}, \cdots, a_n\right), \beta$ $:=(\underbrace{0, \cdots, 0}_{i-1}, 1, \underbrace{0, \cdots, 0}_{n-i})$, and define $\ell: I \rightarrow I^{n}$ by $\ell(t)=\alpha+t \beta$, then by assumption $f\left(a_{1}, \cdots, a_{i-1}, t,\right.$ $\left.a_{i+1},\cdots, a_{n}\right)=f \circ \ell(t)$ is linear. Therefore we have that $f=f\left(x_{1}, \cdots, x_{n}\right)$ is linear in $x_{i}$ for any $i=1, \cdots, n$. By Lemma \ref{lemma1_12}. $f\in\Xi_n$, and in particular $f$ is a polynomial.
    
    Provided that $\operatorname{deg} f\geqslant 2$, then define $\ell_{\Delta}: I \rightarrow I^{n}$ by $\ell_{\Delta}(t)=\left(\dfrac{t}{\sqrt{n}}, \cdots, \dfrac{t}{\sqrt{n}}\right)$, then $\ell_{\Delta}$ is an isometric embedding but $f \circ \ell_{\Delta}(t)=f\left(\dfrac{t}{\sqrt{n}}, \cdots, \dfrac{t}{\sqrt{n}}\right)$ is a polynomial in $t$ with degree $\geqslant 2$ and hence nonlinear, contradiction.
    
    Therefore $f$ is a polynomial of $\operatorname{deg} f \leqslant 1$, i.e. a linear function.
\end{proof}

\begin{corollary}\label{coro1_14}
    Let $\Omega$ be a connected open subset of $\mathbb{R}^{n}$ and $f: \Omega \rightarrow \mathbb{R}$ is a nonlinear function, then up to rigid motion and homothety we have $I^{n}\subseteq \Omega$ and $f_{0}(t):=f(0, \cdots, 0, t), t \in I$ is a nonlinear function.
\end{corollary}
\begin{proof}
    Apply Remark \ref{remark1_9} and Lemma \ref{lemma1_13}.
\end{proof}

\begin{lemma}\label{lemma1_15}
    Let $\left\{f_{n}(x)=a_{n} x+b_{n}\right\}_{n \in \mathbb{N}}$ be a sequence of linear functions on $I$, if $\left\{f_{n}\right\}_{n \in \mathbb{N}}$ converges weakly then $\exists\ a, b \in \mathbb{R}$ s.t. $\lim\limits_{n \rightarrow \infty} a_{n}=a, \lim\limits_{n \rightarrow \infty} b_{n}=b$ and $f_{n}(x) \stackrel{w}{\rightharpoonup} f(x)=a x+b$ on $I$ as $n \rightarrow \infty$.
\end{lemma}
\begin{proof}
    Since $\left\{f_{n}\right\}_{n \in \mathbb{N}}$ converges pointwise, we have that $\exists\ a, b \in \mathbb{R}$, s.t. $\lim\limits_{n \rightarrow \infty} a_{n}=\lim\limits_{n \rightarrow \infty}\left(f_{n}(1)-f_{n}(0)\right)$ $=\lim\limits_{n \rightarrow \infty} f_{n}(1)-\lim\limits_{n \rightarrow \infty} f_{n}(0)=: a$ and $\lim\limits_{n \rightarrow \infty} b_{n}=\lim\limits_{n \rightarrow \infty} f_{n}(0)=: b$. Therefore $\forall\ x\in I, \lim\limits_{n \rightarrow \infty} f_{n}(x)=\lim\limits_{n \rightarrow \infty}\left(a_{n} x+b_{n}\right)=x \lim\limits_{n \rightarrow \infty} a_{n}+\lim\limits_{n \rightarrow \infty} b_{n}=a x+b=f(x), \text { i.e. } f_{n} \stackrel{w}{\rightharpoonup} f$.
\end{proof}

\begin{proposition}\label{prop1_16}
    Let $\Omega$ be a connected open subset of $\mathbb{R}^{n}$ and $f: \Omega \rightarrow \mathbb{R}$ is a nonlinear continuous function, then up to rigid motion and homothety we have $I^{n} \subseteq \Omega$ and $\exists\ \delta>0$ s.t. $\forall\ x \in B_{\delta}^{n-1} \subseteq I^{n-1}$ we have $f_{x}(t):=f\left(x^{1}, \ldots x^{n-1}, t\right), t\in I$ is a nonlinear continuous function.
\end{proposition}
\begin{proof}
    By Corollary \ref{coro1_14} up to rigid motion and homothety we have $I^{n} \subseteq \Omega$ and $f_{0}(t):=f(0, \cdots, 0, t),$ $t \in I$ is a nonlinear function.
    
    Provided that $\forall\ \varepsilon>0, \exists\ x \in B_{\varepsilon}^{n-1} \subseteq I^{n-1}$ s.t. $f_{x}(t):=f\left(x^{1}, \cdots, x^{n-1}, t\right), t \in I$ is linear. Take any sequence $\left\{x_{n}\right\}_{n \in \mathbb{N}} \subseteq B_{\varepsilon}$ with $\lim\limits_{n \rightarrow \infty} x_{n}=0$, then by the continuity of $f, f_{x_{n}}(t) \stackrel{w}{\rightharpoonup} f_{0}(t)$. Since $\left\{f_{x_{n}}\right\}_{n \in N}$ is a sequence of linear function, by Lemma \ref{lemma1_15} is linear, contradiction.
    
    Therefore $\exists\ \delta>0$ s.t. $\forall\ x \in B_{\delta}^{n-1} \subseteq I^{n-1}$ we have $f_{x}$ is nonlinear, moreover it is clear that $f_{x}$ is continuous.
\end{proof}

\subsection{The Osculating Spaces and Normal Spaces}
We discuss several important notions related to submanifolds, these notions will be used frequently throughout this paper.

As usual, for a point on a differentiable (or $\left.C^{r}\right)$ submanifold the tangent space at this point is defined to be the image of the differential of the parametrization map at this point.

For example, let a differentiable $d$-submanifold $M$ of $\mathbb{R}^{n}$ be the graph of differentiable map $f: I^{d} \rightarrow \mathbb{R}^{n-d}$ at the vicinity of $p:=\left(x_{0}, f\left(x_{0}\right)\right) \in M$, then $T_p M$ is the $\mathbb{R}$-linear span of the column vectors of $\left[\begin{array}{c}I d_{d} \\ J_{x_{0}} f\end{array}\right]$, i.e. $T_{p} M:=\left\{\left(x,\left(J_{x_0} f\right)x\right)\in \mathbb{R}^{d} \times \mathbb{R}^{n-d} \mid x \in \mathbb{R}^{d}\right\}$. Recall that the first order approximation of $f$ at $x_{0}$ is defined by $f\left(x_{0}\right)+\left(J_{x_{0}} f\right)\left(x-x_{0}\right)$, i.e. the linear part of the Taylor expansion of $f$. The graph of the first order approximation is $p+T_{p} M$ (here $+$ is the Minkowski sum), i.e. the affine $d$-subspace of $\mathbb{R}^{n}$ tangent to $M$ at $p$. As usual, the normal space $N_{p} M$ is defined to be the orthogonal complement of $T_pM$.

Recall that for a twice differentiable $d$-submanifold $M$ of $\mathbb{R}^{n}$, the second osculating space at $p \in M$ is defined to be $T_{p}^{2} M:=\operatorname{span}_{\mathbb{R}}\left\{\dfrac{\partial^{2}}{\partial x^{i} \partial x^{j}}, \dfrac{\partial}{\partial x^{k}}\right\}$ where $x^1,\cdots,x^d$ are local coordinates at the vicinity of $p$. An elementary example of the second osculating space is the osculating planes of a $C^{\infty}$ regular curve in $\mathbb{R}^{3}$. Also recall that, if $\exists\ p \in M$ s.t. $\operatorname{dim} T_{p}^{2} M=n$, then $M$ is substantial. This provides an practical criterion of substantialness.

For a $C^{2}$ submanifold $M$, as usual we say that $p \in M$ is a inflection point of $M$ iff $p+T_{p} M$ is of high order tangency with $M$ at $p$, i.e. at the vicinity of $p$, $M$ coincides with the graph of a $C^{2}$ map with a non-Morse singularity at $p$.

\subsection{Embedded Vector Bundles}
For any integer $0 \leqslant k \leqslant n$ we denote by $\operatorname{Gr}_{k}(n)$ the Grassmannian of $k$-subspaces in $\mathbb{R}^{n}$. We denote by $\tau_{k}(n):=\left\{(p, u) \in \operatorname{Gr}_{k}(n) \times \mathbb{R}^{n} \mid u \in p\right\}$ then tautological bundle over $\operatorname{Gr}_{k}(n)$, where the first projection $\pi_{1}$ gives the bundle map, and the second projection $\pi_{2}$ is the blow up.

Let $X$ be a topological subspace of $\mathbb{R}^{n}$ then a continuous map $\varphi: X \rightarrow \operatorname{Gr}_{k}(n)$ pulls back an $k$-vector bundle over $X$ with a canonical embedding into $\mathbb{R}^{n} \times \mathbb{R}^{n}$, by the following commutative diagram.

\begin{figure}[H]
    \centering 
    \begin{tikzpicture}[node distance = 1.5cm]
        \node at (0,0) {$\varphi^{*} \tau_{k}(n)$};
        \node at (-3,0) {$\mathbb{R}^{n} \times \mathbb{R}^{n}$};
        \node at (3,0) {$\tau_{k}(n)$};
        \node at (-3,3) {$\mathbb{R}^{n}$};
        \node at (0,3) {$\mathbb{R}^{n}$};
        \node at (3,3) {$\mathbb{R}^{n}$};
        \node at (-3,-3) {$\mathbb{R}^{n}$};
        \node at (0,-3) {$X$};
        \node at (3,-3) {$\operatorname{Gr}_{k}(n)$};
        \draw [dashed, thick, ->] (0.9,0)--(2.1,0);
        \draw [thick] (-0.8,0) arc (-90:90:0.1);
        \draw [thick, ->] (-0.8,0)--(-2.1,0);
        \draw [thick, ->] (-0.8,3)--(-2.1,3);
        \draw [thick, ->] (0.9,3)--(2.1,3);
        \draw [thick, ->] (-0.8,-3)--(-2.1,-3);
        \draw [thick, ->] (0.9,-3)--(2.1,-3);
        \draw [thick, ->] (-3,0.5)--(-3,2.6);
        \draw [thick, ->] (0,0.5)--(0,2.6);
        \draw [thick, ->] (3,0.5)--(3,2.6);
        \draw [thick, ->] (-3,-0.5)--(-3,-2.6);
        \draw [dashed, thick, ->] (0,-0.5)--(0,-2.6);
        \draw [thick, ->] (3,-0.5)--(3,-2.6);
        \node at (-1.5,0.4) {$i$};
        \node at (1.5,0.4) {$\widetilde{\varphi}$};
        \node at (-1.5,-2.6) {$i_X$};
        \node at (1.5,-2.6) {$\varphi$};
        \node at (-1.5,3.4) {$Id_n$};
        \node at (1.5,3.4) {$Id_n$};
        \node at (-3.8,1.5) {$0\times Id_n$};
        \node at (-3.8,-1.5) {$Id_n\times 0$};
        \node at (3.4,1.5) {$\pi_{2}$};
        \node at (3.4,-1.5) {$\pi_{1}$};
    \end{tikzpicture}
    \caption{}
    \label{fig1}
\end{figure}

$\varphi^{*} \tau_{k}(n)$ is said to be the vector bundle defined by $\varphi$. If $k=1$ then $\operatorname{Gr}_{1}(n)=\mathbb{P}^{n-1}$ and $\varphi^{*} \tau_{1}(n)$ is called a line bundle.

In this paper we consider only embedded vector bundles pullbacked from tautological bundles, by abuse of notion we write vector bundles as abbrevation.

Let us recall several properties and examples of vector bundles.

If the base space $X=M$ is a topological $d$-submanifold of $\mathbb{R}^{n}$ then $\varphi^{*} \tau_{k}(n)$ is a topological $(d+k)$-submanifold of $\mathbb{R}^{2 n}$. Moreover if $M$ is $C^{r}$ and $\varphi$ is $C^{r}$ then $\varphi^{*} \tau_{k}(n)$ is a $C^{r}$ submanifold. An open cover of $M$ simutanously serves as an atlas of $M$ and trivializes $\varphi^{*} \tau_{k}(n)$ exists and is said to be a trivializing atlas, its elements are called trivializing charts. Moreover, take any trivializing chart $U \subseteq M$ we have a parametrization map $\mathbf{f}: \Omega \rightarrow \mathbb{R}^{n}$ and a matrix-valued function $A: \Omega \rightarrow \mathbb{R}^{n \times k}$ s.t. $\mathbf{f}(\Omega)=U, \operatorname{Im} A(x)=\varphi(\mathbf{f}(x))$ and $F: \Omega \times \mathbb{R}^{k} \rightarrow E \subseteq \mathbb{R}^{n} \times \mathbb{R}^{n}$ defined by $F(x, \xi) \mapsto(\mathbf{f}(x), A(x) \xi)$ is the trivilization over $U$.

Let $f: Y \rightarrow X$ be a continuous map between two topological subspaces of $\mathbb{R}^{n}$ and $E:=\varphi^{*} \tau_{k}(n)$ be a $k$-vector bundle over $X$ defined by continuous map $\varphi: X \rightarrow \operatorname{Gr}_{k}(n)$, then $f$ pulls back a $k$-vector bundle $f^{*} E=f^{*}\left(\varphi^{*} \tau_{k}(n)\right)=(\varphi \circ f)^{*} \tau_{k}(n)=\{(y, u) \mid u \in \varphi(f(y)), y \in Y\}$ over $Y$. In particular if $Y \stackrel{i}{\subseteq} X$ then $i^{*} E=\{(y, u) \mid u \in \varphi(y), y \in Y\} \subseteq E$.

For any differentiable submanifold $M$ in $\mathbb{R}^{n}$ we have the subset $T M:=\left\{(p, u) \mid u \in T_{p} M, p \in M\right\} \subseteq \mathbb{R}^{n} \times \mathbb{R}^{n}$. Notice that if $M$ is $C^{r}$ with positive $r$ then $T M$ is the tangent bundle defined by $C^{r-1}$ map $\varphi: p \rightarrow T_{p} M$. If $M$ is differentiable but not $C^{1}$ then $T M$ does not necessarily have a bundle structure. By abuse of notation in this case we still call it the tangent bundle of $M$. 

The case for the normal bundle $N M:=\left\{(p, u) \mid u \in N_{p} M, p \in M\right\}$ is analogous.

\subsection{The Codiagonal Morphism}
In the category of finite-dimensional $\mathbb{R}$-vector spaces the codiagonal morphism $\Delta_{*}$ is defined via the following commutative diagram as the dual of $\Delta$.
\begin{figure}[H]
    \centering
    \begin{tikzpicture}
        \node at (0,0) {$\mathbb{R}^{n} \oplus \mathbb{R}^{n}$};
        \node at (3,0) {$\mathbb{R}^{n}$};
        \node at (-3,0) {$\mathbb{R}^{n}$};
        \draw [thick, ->] (-2.5,0)--(-1,0);
        \draw [thick] (-2.5,0) arc (270:90:0.1);
        \draw [thick, ->] (2.5,0)--(1,0);
        \draw [thick] (2.5,0) arc (-90:90:0.1);
        \node at (-1.8,-0.3) {$i_1$};
        \node at (1.8,-0.3) {$i_2$};
        \node at (0,3) {$\mathbb{R}^{n}$};
        \draw [dashed, thick, ->] (0,0.5)--(0,2.5);
        \draw [thick, ->] (-2.8,0.5)--(-0.3,2.7);
        \draw [thick, ->] (2.8,0.5)--(0.3,2.7);
        \node at (0.5,1.5) {$\Delta_{*}$};
        \node at (2.1,1.7) {$Id_n$};
        \node at (-2.1,1.7) {$Id_n$};
    \end{tikzpicture}
    \caption{}
    \label{fig2}
\end{figure}

As a map on Euclidean spaces, $\Delta_{*}: \begin{aligned}&\mathbb{R}^{n} \times \mathbb{R}^{n} \rightarrow \mathbb{R}^{n} \\&(x, y) \mapsto x+y\end{aligned}$ is bi-linear and hence $C^{\infty}$. Take any subset $E$ of $\mathbb{R}^{n} \times \mathbb{R}^{n}, \Delta_{*} E \subseteq \mathbb{R}^{n}$ is the continuous image under $\Delta_{*}$. It possesses the following obvious properties:
\begin{enumerate}
    \item[i).] Take any subsets $A, B \subseteq \mathbb{R}^{n} \times \mathbb{R}^{n}$, we have $\Delta_{*}(A \cup B)=\Delta_{*} A \cup \Delta_{*} B$ and $\Delta_{*}(A \cap B)=\Delta_{*} A \cap \Delta_{*} B$. If $A \subseteq B$, then $\Delta_{*} A \subseteq \Delta_{*} B$.
    \item[ii).] Let $E$ be a vector bundle defined by $\varphi: X \rightarrow \operatorname{Gr}_{k}(n)$, then $\Delta_{*} E=\bigcup\limits_{x \in X}(x+\varphi(x))$. In particular, if $M$ is a differentiable submanifold, then $\Delta_{*} T M=\bigcup\limits_{p \in M}\left(p+T_{p} M\right)$ and $\Delta_{*} N M=\bigcup\limits_{p \in M}\left(p+N_{p} M\right)$.
    \item[iii).] Take any differentiable submanifolds $M \subseteq \mathbb{R}^{m}$ and $N \subseteq \mathbb{R}^{n}$, we have $\Delta_{*} T(M \times N)=\Delta_{*}\left\{((p, q),(u, v)) \mid (u, v) \in T_{p} M \oplus T_{q} N, (p, q) \in M \times N\right\}=\Delta_{*}\{((p, q),(u, v)) \mid u \in T_{p} M, v$
    $\in T_{q} N, p \in M, q \in N\}=\{(p+u, q+v) \mid u \in \left.T_{P} M, v \in T_{q} N,\right.$
    $\left.p \in M, q \in N\right\}=\Delta_{*} T M \times \Delta_{*} T N$.
\end{enumerate}

Let $E$ be a $k$-vector bundle over a topological subspace $X \subseteq \mathbb{R}^{n}$ then $\Delta_{*} E$ is the union of a family of $k$-affine subspaces of $\mathbb{R}^{n}$ covers $X$. For example if $X$ is a $C^{\infty}$ regular curve in $\mathbb{R}^{3}$ and $E=T X$ is its tangent bundle, then $\Delta_{*} E$ is the tangent developable of $X$. However, as we will show later in this paper, if the assumption on smoothness is weakened then the struture of $\Delta_{*} E$ can be complicated. In geometric measure theory, it is important to study these objects from analytic point of view, since they are closely related to Kakeya's conjecture. 

For instance, Cumberbatch, Keleti, and Zhang \cite{6} studied the Hausdorff dimension of the union of tangent lines of a plane curve and in particular, they constructed a convex curve whose union of one-sided tangents has small Hausdorff dimension, these results were generalized to high dimensions in \cite{5}. R.O. Davies \cite{7} proved that any subset $A$ of the plane can be covered by a collection of affine lines such that the union of lines has the same Lebesgue measure as $A$, and in \cite{8} Esa Järvenpää et al. studied the non-empty interior of continuously parametrized Besicovitch sets. Also, in \cite{9}, Korn{\'e}lia H{\'e}ra,  {Keleti Tam{\'a}s} and {Andr{\'a}s M{\'a}th{\'e}} estimated the Hausdorff dimension of union of family of affine spaces.

In this paper we mainly discuss the subcase that $X=M$ is a (at least) differentiable $d$-submanifold of $\mathbb{R}^{n}$ and $E$ is a $k$-vector bundle or $E=T M, N M$. Under this assumption, $E$ can be regarded as a parametrized family of $k$-affine subspaces with $d$ independent parameters, hence we expect $\Delta_{*} E$ to have Hausdorff dimension $k+d$ when $k+d \leqslant n$, and have positive $n$-dimensional Lebesgue measure when $k+d>n$. In this case we (informally) say that $E$ attains its expected size. We shall study in this paper the conditions for normal and tangent bundles, and then general $k$-vector bundles and line bundles, to attain expected size.

\subsection{Section Transform for Differentiable Mappings}
We introduce tools for studying tangent bundle of graph of a differentiable map.
\begin{definition}\label{defin1_17}
    For any differentiable map $f: I^{d} \rightarrow \mathbb{R}^{n-d}$, we define the lifted section map of $f$ by
    $$
    \Phi_{f}: \begin{aligned}
    &I^{d} \times \mathbb{R}^{d} \longrightarrow \mathbb{R}^{d} \times \mathbb{R}^{n-d} \\
    &(x, a) \longmapsto\left(a, f(x)+\left(J_{x} f\right)(a-x)\right)
    \end{aligned}
    $$
    and for any $a \in \mathbb{R}^{d}$ define the section transform of $f$ at $a$ by
    $$
    \phi_{f}^{a}: \begin{aligned}
    I^d & \longrightarrow \mathbb{R}^{n-d} \\
    x & \mapsto f(x)+\left(J_{x} f\right)(a-x)
    \end{aligned}
    $$
\end{definition}

The importance of the lifted section map is suggested by the following lemma.

\begin{lemma}\label{lemma1_18}
    For any differentiable map $f: I^{d} \rightarrow \mathbb{R}^{n-d}$ we have $\Phi_{f}\left(I^{d} \times \mathbb{R}^{d}\right)=\Delta_{*} T \Gamma_{f}$.
\end{lemma}
\begin{proof}
    Since $T \Gamma_{f}$ is parametrized globally by
    $$
    F: \begin{aligned}
    &I^{d} \times \mathbb{R}^{d} \longrightarrow \mathbb{R}^{n} \times \mathbb{R}^{n} \\
    &(x, \xi) \longmapsto\left((x, f(x)),\left[\begin{array}{l}
    \operatorname{Id}_{d} \\
    J_{x} f
    \end{array}\right] \xi\right)
    \end{aligned}
    $$
    in particular $T \Gamma_{f}=F\left(I^{d} \times \mathbb{R}^{d}\right)$ and hence $\Delta_{*} T \Gamma_{f}=\Delta_{*} F\left(I^{d} \times \mathbb{R}^{d}\right)$. Since $\Delta_{*} F(x, \xi)=\left(x+\xi,\right.$ $\left.f(x)+\left(J_{x} f\right) \xi\right)=\left(x+\xi, f(x)+\left(J_{x} f\right)((x+\xi)-x)\right)=\Phi_{f}(x, x+\xi)$ and $\left\{\begin{array}{l}x=x \\ a=x+\xi\end{array}\right.$ is a change of coordinates in $I^{d} \times \mathbb{R}^{d}$, we conclude that $\Phi_{f}\left(I^{d} \times \mathbb{R}^{d}\right)=\Delta_{*} F\left(I^{d} \times \mathbb{R}^{d}\right)=\Delta_{*} T \Gamma_{f}$.
\end{proof}

In order to compute the lifted section map, we will use several properties of the section transform.

\begin{lemma}\label{lemma1_19}
    For any $a \in \mathbb{R}^{d}$ and differentiable function $f, g: I^{d} \rightarrow \mathbb{R}^{n-d}$ we have:
    \begin{enumerate}
        \item[i).] For any $A \in \operatorname{End}\left(\mathbb{R}^{n-d}\right), \phi_{A f+g}^{a}=A \phi_{f}^{a}+\phi_{g}^{a}$.
        \item[ii).] For any affine map $\ell: I^{k} \rightarrow I^{d}, \phi_{f \circ \ell}^{a}=\phi_{f}^{\ell(a)} \circ \ell$.
    \end{enumerate}
\end{lemma}
\begin{proof}
    \hspace{-10pt}
    \begin{enumerate}
        \item[i).] $\forall\ x \in I^{d}, \phi_{A f+g}^{a}(x)=A f(x)+g(x)+\left(J_{x}(A f+g)\right)(a-x)=A\left(f(x)+\left(J_{x} f\right)(a-x)\right)+\left(g(x)+\left(J_{x} g\right)(a-x)\right)=A \phi_{f}^{a}+\phi_{g}^{a}$.
        \item[ii).] Denote $\ell(0)=\xi$ then $\exists\ A \in \operatorname{End} \left(\mathbb{R}^{n-d}\right)$ s.t $\ell=\xi+A. \forall\ x \in I^{d}, \phi_{f \circ \ell}^{a}(x)=f \circ \ell(x)+J_{x}(f \circ \ell)(a-x)=f(\ell(x))+\left(J_{\ell(x)} f \circ J_{x} \ell\right)(a-x)=f(\ell(x))+\left(J_{\ell(x)} f\right) A(a-x)=f(\ell(x))+\left(J_{\ell(x)} f\right)(\ell(a)-\ell(x))=\phi_{f}^{\ell(a)} \circ \ell(x)$.
    \end{enumerate}
\end{proof}

Before we close this section, we present a weak version of implict function theorem proved by using Brouwer's fixed point theorem instead of the Banach fixed point theorem. The corollary of this theorem is poweful when applied to the lifted section maps.

In the following proof, as usual, for any bounded operator $A$ we denote by $\|A\|_{o p}$ its operator norm and by $A_{L}^{-1}$ its left inverse (if exists).

\begin{theorem}\label{theo1_20}
    Let $0 \leqslant n \leqslant m$ be integers, $\Omega$ be an open subset of $\mathbb{R}^{m}$ and $h: \Omega \rightarrow \mathbb{R}^{n}$ be a continuous map differentiable at $x_{0} \in \Omega$ s.t. $d_{x_{0}} h$ has full rank, then $\exists\ \varepsilon>0$ s.t. $\forall\ y \in \bar{B}_{\lambda \varepsilon}^{n}\left(y_{0}\right), h^{-1}(\{y\}) \cap \bar{B}_{\varepsilon}^{m}\left(x_{0}\right) \neq \varnothing$, and $\forall x \in h^{-1}(\{y\}) \cap \bar{B}_{\varepsilon}^{m}\left(x_{0}\right), \lambda\left\|x-x_{0}\right\| \leqslant\left\|y-y_{0}\right\| \leqslant\left(\left\|d_{x_{0}} h\right\|_{o p}+\lambda\right)\left\|x-x_{0}\right\|$, where $y_{0}:=h\left(x_{0}\right)$ and $\lambda:=\dfrac{1}{2\left\|\left(d_{x_{0}} h\right)_{L}^{-1}\right\|_{o p}}$.
\end{theorem}
\begin{proof}
    By the definition of derivative $\exists$ sufficiently small $\varepsilon>0$, s.t. $\bar{B}_{\varepsilon}^{m}\left(x_{0}\right) \subseteq \Omega$ and $\forall\ x \in \bar{B}_{\varepsilon}^{m}\left(x_{0}\right)$, $\left\|\left(h(x)-h\left(x_{0}\right)\right)-\left(d_{x_{0}} h\right)\left(x-x_{0}\right)\right\|<\lambda\left\|x-x_{0}\right\|$. For any $y \in \mathbb{R}^{n}$ define $\Psi_{y}: \Omega \rightarrow \mathbb{R}^{n}$ by $\Psi_{y}:=x-\left(d_{x_{0}} h\right)_{L}^{-1}(h(x)-y)$, then $\Psi_{y}$ is continuous, and $\Psi_{y}(x)=x$ iff $\left(d_{x_{0}} h\right)_{L}^{-1}(h(x)-y)=0$, i.e. $h(x)=y$. Moreover we have 
    $$
    \begin{aligned}
    x_{0}-\Psi_{y}(x) &=x_{0}-x+\left(d_{x_{0}} h\right)_{L}^{-1}(h(x)-y) \\
    &=\left(x_{0}-x\right)+\left(d_{x_{0}} h\right)_{L}^{-1}\left(h(x)-h\left(x_{0}\right)\right)+\left(d_{x_{0}} h\right)_{L}^{-1}\left(y_{0}-y\right) \\
    &=\left(d_{x_{0}} h\right)_{L}^{-1}\left(\left(h(x)-h\left(x_{0}\right)\right)-\left(d_{x_{0}} h\right)\left(x-x_{0}\right)\right)+\left(d_{x_{0}} h\right)_{L}^{-1}\left(y_{0}-y\right)
    \end{aligned}
    $$
    Fix $y \in \bar{B}_{\lambda \varepsilon}^{n}\left(y_{0}\right)$. For any $x \in \bar{B}_{\varepsilon}^{m}\left(x_{0}\right)$, we have
    \begin{equation}\label{eq1}\tag{*}
        \begin{aligned}
        \left\|x_{0}-\Psi_{y}(x)\right\| & \leqslant\left\|\left(d_{x_{0}} h\right)_{L}^{-1}\right\|_{o p}\left\|\left(h(x)-h\left(x_{0}\right)\right)-\left(d_{x_{0}} h\right)\left(x-x_{0}\right)\right\|+\left\|\left(d_{x_{0}} h\right)_{L}^{-1}\right\|_{o p}\left\|y-y_{0}\right\| \\
        & \leqslant \dfrac{1}{2 \lambda} \cdot \lambda\left\|x-x_{0}\right\|+\dfrac{1}{2 \lambda}\left\|y-y_{0}\right\|\\ 
        & \leqslant \dfrac{1}{2}\left(\left\|x-x_{0}\right\|+\dfrac{1}{\lambda}\left\|y-y_{0}\right\|\right)\\ 
        &\leqslant \dfrac{1}{2}\left(\varepsilon+\dfrac{1}{\lambda} \cdot \lambda \varepsilon\right)\\ 
        & =\varepsilon
        \end{aligned}
        \end{equation}
        i.e. $\Psi_{y}(x) \in \bar{B}_{\varepsilon}^{m}\left(x_{0}\right)$. Since $\Psi_{y}(x)$ is continuous on $\bar{B}_{\varepsilon}^{m}\left(x_{0}\right)$ and $\Psi_{y}\left(\bar{B}_{\varepsilon}^{m}\left(x_{0}\right)\right) \subseteq \bar{B}_{\varepsilon}^{m}\left(x_{0}\right)$, by Brouwer's fixed point theorem $\exists\ x_{*} \in \bar{B}_{\varepsilon}^{m}\left(x_{0}\right) \quad$ s.t. $\quad \Psi_{y}\left(x_{*}\right)=x_{*}$, i.e. $h\left(x_{*}\right)=y$. Therefore $x_{*} \in h^{-1}(\{y\}) \cap \bar{B}_{\varepsilon}^{m}\left(x_{0}\right)$ and hence $h^{-1}(\{y\}) \cap \bar{B}_{\varepsilon}^{m}\left(x_{0}\right) \neq \varnothing$. 
        
        Take any $x \in h^{-1}(\{y\}) \cap \bar{B}_{\varepsilon}^{m}\left(x_{0}\right)$, then $h(x)=y$ and hence $\Psi_{y}(x)=x$. Therefore by the estimation (\ref{eq1}), we have $\left\|x-x_{0}\right\|=\left\|x_{0}-\Psi_{y}(x)\right\| \leqslant \dfrac{1}{2}\left(\left\|x-x_{0}\right\|+\dfrac{1}{\lambda}\left\|y-y_{0}\right\|\right)$, and hence $\left\|x-x_{0}\right\| \leqslant \dfrac{1}{\lambda}\left\|y-y_{0}\right\|$, i.e. $\lambda\left\|x-x_{0}\right\| \leqslant\left\|y-y_{0}\right\|$. Moreover, we have
        $$
        \begin{aligned}
        \left\|y-y_{0}\right\| &=\left\|h(x)-h\left(x_{0}\right)\right\| \\
        & \leqslant\left\|h(x)-h\left(x_{0}\right)-\left(d_{x_{0}} h\right)\left(x-x_{0}\right)\right\|+\left\|\left(d_{x_{0}} h\right)\left(x-x_{0}\right)\right\| \\
        & \leqslant \lambda\left\|x-x_{0}\right\|+\left\|d_{x_{0}} h\right\|_{o p}\left\|x-x_{0}\right\| \\
        &=\left(\left\|d_{x_{0}} h\right\|_{o p}+\lambda\right)\left\|x-x_{0}\right\|
        \end{aligned}
        $$
        In summary $\lambda\left\|x-x_{0}\right\| \leqslant\left\|y-y_{0}\right\| \leqslant\left(\left\|d_{x_{0}} h\right\|_{o p}+\lambda\right)\left\|x-x_{0}\right\|$. 
        
        This concludes the proof.
\end{proof}

\begin{remark}\label{remark1_21}
    In \cite{4} this theorem is proved for $m=n$ and $\Omega=\mathbb{R}^{n}$. The proof is essentially the same. Notice that the classical implicit function theorem holds also for Banach manifolds. However, since balls in infinite dimensional spaces are not compact in the norm topology, Brouwer's fixed point theorem does not hold, and hence this weak form of the implicit function theorem cannot be generalized directly to general Banach spaces or Hilbert spaces.
\end{remark}

\begin{corollary}\label{coro1_22}
    Let $0 \leqslant n \leqslant m$ be integer, $\Omega$ be an open subset of $\mathbb{R}^{m}$ and $h: \Omega \rightarrow \mathbb{R}^{n}$ be a continuous map. If $\exists\ x_{0} \in \Omega$ s.t. $h$ is differentiable at $x_{0}$ and $d_{x_{0}} h$ has full rank, then $\operatorname{int}(h(\Omega)) \neq \varnothing$.
\end{corollary}
\begin{proof}
    Denote $y_{0}:=h\left(x_{0}\right)$ then by Theorem \ref{theo1_20} $\exists\ \delta>0$ s.t. $\forall\ y \in \bar{B}_{\delta}^{n}\left(y_{0}\right)$ the fiber $h^{-1}(\{y\}) \neq \varnothing$. Therefore $h(\Omega) \supseteq  \bar{B}_{\delta}^{h}(y)$ has non-empty interior.
\end{proof}

\subsection{Universal Measurability and Analyticity}
Before discussing Lebesgue measure and Hausdorff dimension of union of tangents and normals or codiagonal of vector bundles over submanifolds we shall prove their measurability. Using the section transforms we will show that these objects are analytic sets (aka. Suslin sets) in $\mathbb{R}^{n}$, and hence universally measurable. (for basic properties of analytic sets, see e.g. \cite{10} )

Firstly we point out that this property clearly holds for codiagonal of vector bundle over submanifolds. For completeness we give a proof.
\begin{proposition}\label{prop_1_8_1}
    Let $X$ be a topological submanifold (with boundary) of $\mathbb{R}^{n} \times \mathbb{R}^{n}$, then $\Delta_{*} X$ is analytic and hence universally measurable.
\end{proposition}

\begin{proof}
    By the Lindel\"{o}f property of second countable spaces, $X$ admits a countable atlas $\left\{U_{k}\right\}_{k \in \mathbb{N}^{*}}$. Since for any $i \in \mathbb{N}^{*}$, $U_{k}$ is homeomorphic to a open disk, we have that $U_{k}$ is analytic. Therefore $X=\bigcup\limits_{k \in \mathbb{N}^{*}} U_{k}$ is an analytic set. Since $\Delta_{*}$ is continuous, we have that $\Delta_{*} X$ is analytic and hence universally measurable.
\end{proof}

\begin{corollary}\label{coro_1_8_2}
    Let $M$ be a topological submanifold (with boundary) of $\mathbb{R}^{n}$ and $E$ be a vector bundle over $M$, then $\Delta_{*} E$ is analytic and hence universally measurable.
\end{corollary}

Now we shall prove the same assertion for union of tangents (and normals).

\begin{lemma}\label{lemma_1_8_3}
    For any differentiable map $f: I^{d} \rightarrow \mathbb{R}^{n-d}$ we have that $\Delta_{*} T \Gamma_{f}$ is analytic and hence universally measurable.
\end{lemma}
\begin{proof}
    Denote by $e_{1}, \cdots, e_{d}$ the standard orthonormal basis of $\mathbb{R}^{d} \supseteq I^{d}$ and define $\rho_{i}: I^{d} \rightarrow \mathbb{R}_{+}$ by $\rho_{i}(x):=\dfrac{1}{4}\left(\left||\pi-x^{i}|+|\pi+x^i|\right|-\left||\pi-x^{i}|-|\pi-x^{i}|\right|\right)$ then by construction $\forall\ t \in[0,1], \forall\ x \in I^{d}$ we have $x+t \rho_{i}(x) e_{i} \in I^{d}$. For any $k \in \mathbb{N}^{*}, i=1, \cdots, d$ and $j=1, \cdots, n-d$, define $T_{i j}^{(k)}: I^{d} \rightarrow \mathbb{R}^{n-d}$ by $T_{i j}^{(k)}(x):=\dfrac{f^{j}\left(x+\frac{1}{k} \rho_{i}(x) e_{i}\right)-f^{j}(x)}{\rho_{i}(x) / k}$, then $T_{i j}^{(k)}$ is continuous and by the definition of the directional derivatives $\forall\ x \in I^{d}$ we have $\lim\limits_{k \rightarrow \infty} T_{i j}^{(k)}(x)=\dfrac{\partial f^{j}}{\partial x^{i}}(x)$. Therefore the matrix-valued function $\left(T_{i j}^{(k)}\right): I^{d} \rightarrow \mathbb{R}^{(n-d) \times d}$ is continuous and $\forall\ x\in I^d$ we have $\lim\limits_{k \rightarrow \infty}\left(T_{i j}^{(k)}(x)\right)=J_{x} f$.
    
    Define $\Phi_{f}^{(k)}: I^{d} \times \mathbb{R}^{d} \longrightarrow \mathbb{R}^{d} \times \mathbb{R}^{n-d}$ by $\Phi_{f}^{(k)}(x, a):=\left(a, f(x)+\left(T_{i j}^{(k)}(x)\right)(a-x)\right)$, then clearly $\Phi_{f}^{(k)}$ is continuous, and in particular is Borel measurable. Since $\forall\ (x, a) \in I^{d} \times \mathbb{R}^{d}$, $\lim\limits_{k \rightarrow \infty} \Phi_{f}^{(k)}(x, a)=\left(a, f(x)+\lim\limits_{k \rightarrow \infty}\left(T_{i j}^{(k)}(x)\right)(a-x)\right)=\left(a, f(x)+\left(J_{x} f\right)(a-x)\right)=\Phi_{f}(x, a)$, i.e. $\Phi_{f}$ is the pointwise limit of $\left\{\Phi_{f}^{(k)}\right\}_{k \in \mathbb{N}^{*}}$, we conclude that $\Phi_{f}$ is Borel. Therefore $\Gamma_{\Phi_{f}}$ is a Borel set. Since $\Phi_{f}\left(I^{d} \times \mathbb{R}^{d}\right)$ is the projection image of $\Gamma_{\Phi_{f}}$ on to the codomain of $\Phi_{f}$, we obtain that $\Phi_{f}\left(I^{d} \times \mathbb{R}^{d}\right)$ is analytic.
    
    This proves that $\Delta_{*} T \Gamma_{f}=\Phi_{f}\left(I^{d} \times \mathbb{R}^{d}\right)$ is analytic and hence universally measurable.
\end{proof}

\begin{proposition}\label{prop_1_8_4}
    Let $M$ be a differentiable submanifold then $\Delta_{*} T M$ is analytic and hence universally measurable.
\end{proposition}
\begin{proof}
    By the definition of differentiable submanifolds, and by Corollary \ref{coro1_11}, $\exists$ atlas $\left\{U_{k}\right\}_{k \in \Lambda}$ of $M$ s.t. $U_{k}=\Gamma_{f_{k}}$ where up to a rigid motion and homothety $f_{k}: I^{d} \rightarrow \mathbb{R}^{n-d}$ is differentiable. Moreover, by the Lindel\"{o}f property wlog we assume that $\Lambda \equiv \mathbb{N}^{*}$ is countable. Since $\Delta_{*} T M=\Delta_{*} T\left(\bigcup\limits_{k \in \mathbb{N}^{*}} U_{k}\right)=\Delta_{*}\left(\bigcup\limits_{k \in \mathbb{N}^{*}} T U_{k}\right)=\bigcup\limits_{k \in \mathbb{N}^{*}}\left(\Delta_{*} TU_{k}\right)$, and by Lemma \ref{lemma_1_8_3}, $\forall\ k \in \mathbb{N}^{*}$, $\Delta_{*} T U_{k}=\Delta_{*} T \Gamma_{f_{k}}$ is analytic, we conclude that $\Delta_{*} T M$ is analytic, and hence universally measurable.
\end{proof}

\begin{lemma}\label{lemma_1_8_5}
    For any differentiable map $f: I^{d} \rightarrow \mathbb{R}^{n-d}$ we have that $\Delta_{*} N \Gamma_{f}$ is analytic and hence universally measurable.
\end{lemma}
\begin{proof}
    Use the same notations as in Lemma \ref{lemma_1_8_3}.
    
    Take any $x_{0} \in I^{d}$ and denote $p:=\left(x_{0}, f\left(x_{0}\right)\right)$ then it is clear that $p+N_{p} \Gamma_{f}=\left\{(x, y) \in \mathbb{R}^{d} \times \mathbb{R}^{n-d} \mid\right.$ $\left.\left(x-x_{0}\right)+\left(J_{x_0} f\right)\left(y-f\left(x_{0}\right)\right)=0\right\}=\left\{(x, y) \in \mathbb{R}^{d} \times \mathbb{R}^{n-d} \mid x=x_{0}+\left(J_{x_0}f\right)\left(f\left(x_{0}\right)-y\right)\right\}=\left\{\left(x_{0}+\right.\right.$ $\left.\left.\left(J_{x_0} f\right)\left(f\left(x_{0}\right)-y\right), y\right) \mid y \in \mathbb{R}^{n-d}\right\}$. Therefore define $\Psi_{f}: I^{d} \times \mathbb{R}^{n-d} \rightarrow \mathbb{R}^{d} \times \mathbb{R}^{n-d}$ by $\Psi_{f}(x, b)=\left(x+\left(J_{x} b\right)(f(x)-b), b\right)$ then $\Psi_{f}\left(I^{d} \times \mathbb{R}^{n-d}\right)=\bigcup\limits_{p \in \Gamma_{f}}\left(p+N_{p} \Gamma_{f}\right)=\Delta_{*} N M$.
    
    For any $k \in \mathbb{N}^{*}$ define $\Psi_{f}^{(k)}: I^{d} \times \mathbb{R}^{n-d} \rightarrow \mathbb{R}^{d} \times \mathbb{R}^{n-d}$ by $\Psi_{f}^{(k)}(x, b):=\left(x+\left(T_{i j}^{(k)}(x)\right)(f(x)-b), b\right)$ where $\left(T_{ij}^{(k)}\right)$ is the approximation of Jacobian defined in Lemma \ref{lemma_1_8_3}, then similarly we have that $\left\{\Psi_{f}^{(k)}\right\}_{k \in \mathbb{N}^{*}}$ is a sequence of Borel function with pointwise limit $\Psi_{f}$ and hence $\Psi_{f}$ is Borel. Therefore $\Gamma_{\Psi_f}$ is a Borel set, and $\Psi_{f}\left(I^{d} \times \mathbb{R}^{n-d}\right)$ is an analytic set.
    
    This proves that $\Delta_{*} N \Gamma_{f}=\Psi_{f}\left(I^{d} \times \mathbb{R}^{n-d}\right)$ is analytic and hence universally measurable.
\end{proof}

\begin{proposition}\label{prop_1_8_6}
    Let $M$ be a differentiable submanifold then $\Delta_{*} N M$ is analytic and hence universally measurable.
\end{proposition}
\begin{proof}
    Apply Lemma \ref{lemma_1_8_5}, the proof is the same as that of Proposition \ref{prop_1_8_4}.
\end{proof}

Now we are prepared to study the measure theoretical size of union of tangents and normals, and codiagonal of vector bundles over submanifolds.

\section{Size of union of normals and tangents}
\label{sec2}
\subsection{Normal Bundles}
We first study the size of normal bundles. 

Take the tubular neighborhood theorem into account, one may expect that the normal bundle over a smooth submanifold always contains an open subset. Actually more is true: using the square distance function, we will show that this assertion holds for all differentiable submanifolds.

\begin{definition}\label{defi2_1}
    Take any subset $X \subseteq \mathbb{R}^{n}$ and point $q \in \mathbb{R}^{n}$, define the square distance $S_{X}^{q}: X \rightarrow \mathbb{R}$ by $S_{X}^{q}(p):=\|p-q\|^{2}$.
\end{definition}

\begin{lemma}\label{lemma2_2}
    For any differentiable $d$-submanifold $M$ of $\mathbb{R}^{n}, p \in M$ is a critical point of $S_{M}^{q}$ iff $q-p \in N_{p} M$.
\end{lemma}
\begin{proof}
    Let $\mathbf{r}: B_{\varepsilon}^{d} \rightarrow \mathbb{R}^{n}$ be a parametrization map of $M$ at the vicinity of $p \equiv \mathbf{r}(0)$ then $q \in N_{p} M$ iff $\forall\ i \in\{1, \cdots, d\}\left\langle q-\mathbf{r}(0), \dfrac{\partial \mathbf{r}}{\partial x_{i}}(0)\right\rangle=0$, i.e. $\dfrac{\partial}{\partial x_{i}} S_{M}^{q}(\mathbf{r}(0))=2\left\langle q-\mathbf{r}(0), \dfrac{\partial \mathbf{r}}{\partial x_{i}}(0)\right\rangle=0$, i.e. $S_{M}^{q}$ is critical at $\mathbf{r}(0)=p$.
\end{proof}

\begin{corollary}\label{coro2_3}
    Let $M$ be a closed differentiable submanifold of $\mathbb{R}^{n}$, then $\Delta_{*} N M=\mathbb{R}^n$.
\end{corollary}
\begin{proof}
    Take any $q \in \mathbb{R}^{n}$, since $q$ and $M$ are compact, $S_{M}^{q}$ attains a minima $S_{M}^{q}\left(p_{0}\right)$, in particular $S_{M}^{q}$ is critical at $p_{0}$. By Lemma \ref{lemma2_2} $q \in p_0+N_{p_{0}} M \subseteq \Delta_{*} N M$.
\end{proof}

\begin{corollary}\label{coro2_4}
    Let $M$ be a differentiable $d$-submanifold of $\mathbb{R}^{n}$, then $\exists$ open neighbourhood $\nu(M)$ in $\mathbb{R}^{n}$ s.t. $M \subseteq \nu(M) \subseteq \Delta_{*} N M$. In particular, $\operatorname{int}\left(\Delta_{*} N M\right) \neq \varnothing$.
\end{corollary}
\begin{proof}
    Take any $p \in M$, for a sufficiently small $r>0$, we have $B:=B_{2 r}^{n}(p) \cap M$ is homeomorphic to $B_{2 r}^{d}$, and $\bar{B} \backslash B \subseteq \bar{B}_{2 r}^{n}(p) \backslash B_{2 r}^{n}(p)=\partial B_{2 r}^{n}(p)$.

    Take any $q \in B_{r}^{n}(p)$, then $S_{\bar{B}}^{q}$ attains a minima $S_{\bar{B}}^{q}\left(p_{0}\right)$. Since $d(q, \bar{B} \backslash B) \geqslant d\left(q, \partial B_{2r}^{n}(p)\right)>r>d(q, p)$, we have that $p_{0} \in B$. Therefore $p_{0}$ is a critical point of $S_{B}^{q}=\left.S_{\bar{B}}^{q}\right|_{B}$. Since $B$ is a differentiable submanifold of $\mathbb{R}^{n}$, by Lemma \ref{lemma2_2} $q \in p_0+N_{p_0} B \subseteq \Delta_{*} N M$. Therefore $B_{r}^{n}(p) \subseteq \Delta_{*} N M$.

    This proves that $\forall\ p \in M, p \in \operatorname{int}\left(\Delta_{*} N M\right)$. Therefore the required neighborhood exists.
\end{proof}

\subsection{Tangent Bundles}
For the union of tangents, the situation is more complicated.

We start by estimating the Hausdorff dimension.

Firstly, for any differentiable $d$-submanifold $M$ of $\mathbb{R}^{n}$, since $M \subseteq \Delta_{*} T M$, we have the lower bound $d=\operatorname{dim} M \leqslant \operatorname{dim}_{\mathcal{H}} M \leqslant \operatorname{dim}_{\mathcal{H}} \Delta_{*} T M$. By taking $M$ to be an affine subspace, one can show that this inequality is sharp.

Moreover, under several natural conditions, we have the following upper bounds:
\begin{proposition}\label{prop2_5}
    Let $\Pi \subseteq \mathbb{R}^{n}$ be a $d$-affine subspace and $M$ a differentiable submanifold of $\mathbb{R}^{n}$, then $M \subseteq \Pi$ iff $\Delta_{*} T M \subseteq \Pi$. In particular in this case we have $\operatorname{dim}_{\mathcal{H}} \Delta_{*} T M \leqslant d$.
\end{proposition}
\begin{proof}
    Obvious.
\end{proof}

\begin{proposition}\label{prop2_6}
    If $M$ is a Lipschitz-continuously differentiable $d$-submanifold, then $\operatorname{dim}_{\mathcal{H}} \Delta_{*} T M \leqslant 2 d$.
\end{proposition}
\begin{proof}
    Cover $M$ by a countable atlas $\left\{U_{k}\right\}_{k \in \mathbb{N}}$ s.t. $U_{k}=\Gamma_{f_{k}}$ where up to rigid motion and homothety $f_{k}: I^{d} \rightarrow \mathbb{R}^{n-d}$ is differentiable and $d f_{k}$ is locally Lipschitz, then the lifted section map $\Phi_{f_{k}}(x, a)=\left(a, f_{k}(x)+\left(J_{x} f_{k}\right)(a-x)\right)$ is locally Lipschitz, and hence $\operatorname{dim}_{\mathcal{H}} \Phi_{f_{k}}\left(I^{d} \times \mathbb{R}^{d}\right) \leqslant \operatorname{dim}_{\mathcal{H}}\left(I^{d} \times \mathbb{R}^{d}\right)=2 d$. Since $\Phi_{f_{k}}\left(I^{d} \times \mathbb{R}^{d}\right)=\Delta_{*} T U_{k}$ and $\Delta_{*} T M=\bigcup\limits_{k \in \mathbb{N}}\left(\Delta_{*} T U_{k}\right)$, we have that $\operatorname{dim}_{\mathcal{H}} \Delta_{*} T M=\operatorname{dim}_{\mathcal{H}}\left(\bigcup\limits_{k \in N} \Delta_{*} T U_{k}\right)=\sup\limits_{k \in \mathbb{N}} \operatorname{dim}_{\mathcal{H}} \Delta_{*} T U_{k} \leqslant 2 d$.
\end{proof}

\begin{corollary}\label{coro2_7}
    If $M$ is a $C^2$ $d$-submanifold, then $\operatorname{dim}_{\mathcal{H}} \Delta_{*} T M \leqslant 2 d$. 
\end{corollary}

The following example shows that the inequalities in Proposition \ref{prop2_6} and Corollary \ref{coro2_7} are sharp.

\begin{proposition}\label{prop2_8}
    For any $d \leqslant n \in \mathbb{N}^{*}, \exists\ C^{\infty}$ $d$-submanifold $M$ of $\mathbb{R}^{n}$ s.t. $\mathcal{L}^{2 d}\left(\Delta_{*} T M\right)=+\infty$ or $\Delta_{*} T M=\mathbb{R}^{n}$.
\end{proposition}
\begin{proof}
    Identify $\mathbb{C}^{1} \equiv \mathbb{R}^{2}$. Define $\gamma_{1}:=\left\{6 e^{i t} \mid t \in\left[\dfrac{\pi}{2}, \dfrac{3 \pi}{2}\right]\right\}, \gamma_{2}:=\left\{-3 i+3 e^{i t} \mid t \in \left[-\dfrac{\pi}{2}, \dfrac{\pi}{2}\right]\right\}, \gamma_{3}:=\left\{2 i+2 e^{i t} \mid t \in \left[ \dfrac{\pi}{2}, \dfrac{3 \pi}{2}\right]\right\}, \gamma_{4}:=\left\{5i+e^{i t} \mid t \in\left[-\dfrac{\pi}{2}, \dfrac{\pi}{2}\right]\right\}$, then $\gamma:=\bigcup\limits_{k=1}^{4} \gamma_{k}$ glues to a closed $C^{\infty}$ regular curve in $\mathbb{R}^{2}$.

    It is easy to show that $\Delta_{*} T \gamma_{1}=\mathbb{C}^{1} \backslash(\{|z|<1\} \cup\{\operatorname{Re} z \geqslant 0,|\operatorname{Im} z| \leqslant 1\})$ and mutatis mutandis for $\gamma_{2}, \gamma_{3}, \gamma_{4}$, then it is clear that $\Delta_{*} T \gamma=\Delta_{*}\left(\bigcup\limits_{i=1}^{4} T \gamma_{i}\right)=\bigcup\limits_{i=1}^{4} \Delta_{*} T \gamma_{i}=\mathbb{C}^{1} \equiv \mathbb{R}^{2}$.

    Consider the following two cases:

    Case 1: $d<\dfrac{n}{2}$. Define $M:=\underbrace{\gamma \times \cdots \times \gamma}_{d} \subseteq \underbrace{\mathbb{R}^{2} \times \cdots \times \mathbb{R}^{2}}_{d}=\mathbb{R}^{2 d} \hookrightarrow \mathbb{R}^{n}$, then $M$ is a $C^{\infty} d$-submanifold of $\mathbb{R}^{n}$. Since $\Delta_{*} T M=\underbrace{\left(\Delta_{*} T \gamma\right) \times \cdots \times\left(\Delta_{*} T \gamma\right)}_{d}=\underbrace{\mathbb{R}^{2} \times \cdots \times \mathbb{R}^{2}}_{d}=\mathbb{R}^{2 d}$, we have $\mathcal{L}^{2 d}\left(\Delta_{*} T M\right)=+\infty$.

    Case 2: $d \geqslant \dfrac{n}{2}$. Define $M:=\underbrace{\gamma \times \cdots \times \gamma}_{n-d} \times \mathbb{R}^{2 d-n}$, then $M$ is a $C^{\infty}$ submanifold of $\mathbb{R}^{n}$ with $\operatorname{dim} M=n-d+2 d-n=d$, and $\Delta_{*} T M=\underbrace{\mathbb{R}^{2} \times \cdots \times \mathbb{R}^{2}}_{n-d} \times \mathbb{R}^{2 d-n}=\mathbb{R}^{2(n-d)+2 d-n}=\mathbb{R}^{n}$.
\end{proof}

The following example shows that the Lipschitz condition in Proposition \ref{prop2_6} is necessary.

\begin{proposition}\label{prop2_9}
    For any $d \leqslant n \in \mathbb{N}^{*}, \exists\ C^{1}$ $d$-submanifold $M$ of $\mathbb{R}^{n}$, s.t. $\Delta_{*} T M=\mathbb{R}^{n}$.
\end{proposition}
\begin{proof}
    Recall that $\forall\ m \in \mathbb{N}$ we have $\mathbb{R}^{m}$ is a $\sigma$-Peano space, i.e. $\exists$ continuous surjection $\varphi_{m}: \mathbb{R} \rightarrow \mathbb{R}^{m}$. Moreover, $\varphi_{m}$ can be chosen to be a self-similar curve. Denote $\varphi:=\varphi_{n-d+1}$ and consider the following Cauchy Problem:
    \begin{equation}\label{eq2}\tag{*}
        f(t)+f^{\prime}(t)=\varphi(t)
    \end{equation}
    with initial condition $f(0)=0$.

    By Picard's existence theorem, $\exists\ \varepsilon>0$ and differentiable function $\alpha: \left.\right]-\varepsilon, \varepsilon \left[\right. \rightarrow \mathbb{R}^{n-d+1}$ s.t. $\alpha(t)+\alpha^{\prime}(t)=\varphi(t)$ and $\alpha(0)=0$. Since $\alpha^{\prime}=\left.\varphi\right|_{\left.\right]-\varepsilon, \varepsilon \left[\right.}-\alpha$ is continuous, we have that $\alpha$ is $C^{1}$.

    Provided that $\alpha^{\prime} \equiv 0$, then $\alpha \equiv$ Constant and hence $\left.\varphi\right|_{\left.\right]-\varepsilon, \varepsilon \left[\right.}=\alpha+\alpha^{\prime}=\text{Constant}+0 \equiv \text {Constant}$, contracts the self-similarity of $\varphi$. Therefore $\exists\ t_{0} \in \left.\right]-\varepsilon, \varepsilon\left[\right.$ s.t. $\alpha^{\prime}\left(t_{0}\right) \neq 0$, and by the continuity of $\alpha^{\prime}$, $\exists$ open neighborhood $U$ of $t_{0}$ s.t. $\left.\alpha^{\prime}\right|_{U} \neq 0$. Therefore $\gamma:=\left.\alpha\right|_{U}$ is a $C^{1}$ regular curve in $\mathbb{R}^{n-d+1}$.

  It is clear that $\forall\ t \in U, \varphi(t)=\alpha(t)+\alpha^{\prime}(t) \in \Delta_{*} T \gamma$, i.e. $\varphi(U) \subseteq \Delta_{*} T \gamma$. By the self-similarity of $\varphi$, we have that $\varphi(U)=\varphi(\mathbb{R})=\mathbb{R}^{n-d+1}$. This proves that $\Delta_{*} T \gamma=\mathbb{R}^{n-d+1}$.

  Define $M:=\gamma \times \mathbb{R}^{d-1} \hookrightarrow \mathbb{R}^{n-d+1} \times \mathbb{R}^{d-1}=\mathbb{R}^{n}$, then $M$ is a $C^{1}$ $d$-submanifold of $\mathbb{R}^{n}$, and $\Delta_{*} T M=\left(\Delta_{*} T \gamma\right) \times\left(\Delta_{*} T \mathbb{R}^{d-1}\right)=\mathbb{R}^{n-d+1} \times \mathbb{R}^{d-1}=\mathbb{R}^{n}$.
\end{proof}

\begin{remark}\label{remark2_10}
    The phenomenon in Proposition \ref{prop2_9} is anti-intuitive. It shows that tangent lines of a $C^{1}$ regular curve can fill up the ambient space $\mathbb{R}^{n}$, and its tangent bundle is a 2 -dimensional surface embedded in $\mathbb{R}^{n} \times \mathbb{R}^{n}$ with $n$-dimensional codiagonal.
\end{remark}

From the previous discussion we have seen that union of tangents of various submanifolds can be very different in size. Take a differentiable submanifold $M$ we shall provide a sufficient condition for int $\left(\Delta_{*} T M\right) \neq \varnothing$. By Proposition \ref{prop2_5} and Corollary \ref{coro2_7} $M$ should be substantial and $\operatorname{dim} M \geqslant \dfrac{n}{2}$, but the following example shows that these two conditions are not enough.

\begin{proposition}\label{prop2_11}
    Let $g: I \rightarrow \mathbb{R}^{4}$ be a substantial $C^{3}$ regular curve with positive curvature then its one-sided tangent developable $\Sigma_{+}$ is a substantial $C^{2}$ regular surface in $\mathbb{R}^{4}$ s.t. $\mathcal{L}^{4}\left(\Delta_{*} T \Sigma_{+}\right)=0$.
\end{proposition}
\begin{proof}
    By Proposition \ref{prop1_4} $\Sigma_{+}$is a substantial $C^{2}$ regular surface in $\mathbb{R}^4$. 
    
    Since $\Sigma_{+}$ is parametrized by $\mathbf{x}(u, v)=g(u)+v g^{\prime}(u)$ and $\dfrac{\partial \mathbf{x}}{\partial u}(u, v)=g^{\prime}(u)+v g^{\prime \prime}(u), \dfrac{\partial \mathbf{x}}{\partial v}(u, v)=g^{\prime}(u)$, for any $p:=\mathbf{x}(u, v)\in\Sigma_{+}$ we have $p+T_{p} \Sigma_{+}=\left\{\mathbf{x}(u, v)+\lambda \dfrac{\partial \mathbf{x}}{\partial u}(u, v)+\right.$ $\left.\mu \dfrac{\partial \mathbf{x}}{\partial v}(u, v) \mid(\lambda, \mu) \in \mathbb{R}^{2}\right\}$
    $=\left\{g(u)+(\lambda+\mu) g^{\prime}(u)+\lambda v g^{\prime \prime}(u)\mid(\lambda, \mu) \in \mathbb{R}^{2}\right\}$. Define $C^{1}$ function $f: \mathbb{R}^{3} \rightarrow \mathbb{R}^{4}$ by $f(x, y, u):=g(u)+x g^{\prime}(u)+y g^{\prime \prime}(u)$, then $f\left(\mathbb{R}^{3}\right) \supseteq \bigcup\limits_{p \in \Sigma_{+}}(p+T_{p} \Sigma_{+})=\Delta_{*} T \Sigma_{+}$. Since $f$ is $C^{1}$, by Sard's lemma we have $\mathcal{L}^{4}\left(f\left(\mathbb{R}^{3}\right)\right)=0$, and hence $\mathcal{L}^{4}\left(\Delta_{*} T \Sigma_{+}\right)=0$.
\end{proof} 

By strengthening the condition we obtain the following theorems.

\begin{theorem}\label{theo2_12}
    Let $\Omega$ be an open subset in $\mathbb{R}^{d}$ and $f: \Omega \rightarrow \mathbb{R}^{n-1}$ be a $C^{1}$ map s.t. $f$ is twice differentiable at $x_{0} \in \Omega$ with $d_{x_0}^{2} f$ non-degenerated, then $\operatorname{int}\left(\Delta_{*} \Gamma_{f}\right) \neq \varnothing$.
\end{theorem}
\begin{proof}
    Wlog assume that $\Omega=I^{d}$ and $x_{0}=0$. It suffices to prove that $\operatorname{int}\left(\Phi_{f}\left(I^{d} \times \mathbb{R}^{d}\right)\right)\neq \varnothing$.

    For any $b \in \mathbb{R}^{d}$, it is clear that $f(x), J_{x} f$ and hence $\Phi_{f}(x, a)=\left(a, f(x)+\left(J_{x} f\right)(a-x)\right)$ are continuous and differentiable at $(0, b)$, and $\left.\dfrac{\partial \Phi_{f}(x, a)}{\partial(x, a)}\right|_{(0, b)}=\left[\begin{array}{cc}0 & \left(d_{0}^{2} f\right)b \\ I d_{d} & d_{0} f\end{array}\right]=\left[\begin{array}{cc}
        0 & \left(b^{i} \dfrac{\partial f^{k}(0)}{\partial x^{i} \partial x^{j}}\right) \\
        I d_{d} & J_{0} f
    \end{array}\right]$. By Corollary \ref{coro1_22} it remains to proof that $\exists\ b_{0} \in \mathbb{R}^{d}$ s.t. $\left.\dfrac{\partial \Phi_{f}(x, a)}{\partial(x, a)}\right|_{\left(0, b_{0}\right)}$ has full rank, i.e. $\left(b_{0}^{i} \dfrac{\partial f^{k}(0)}{\partial x^{i} \partial x^{j}}\right)$ has full rank.

    Since by assumption $d_{0}^{2} f$ is non-degenerated, there exists $b_{0} \in \mathbb{R}^{d}$ s.t. the contraction $\left(b_{0}^{i} \dfrac{\partial f^{k}(0)}{\partial x^{i} \partial x^{j}}\right)$ has full rank.

    This concludes the proof.
\end{proof}

\begin{theorem}\label{theo2_13} 
    Let $M$ be a twice differentiable $d$-submanifold of $\mathbb{R}^{n}$ s.t. $d \geqslant \dfrac{n}{2}$ and $\exists\ p \in M$ s.t. $\operatorname{dim} T_{p}^{2} M=n$, then $\operatorname{int}\left(\Delta_{*} T M\right) \neq \varnothing$.
\end{theorem}
\begin{proof}
    Wlog assume that $p=0$.

    By assumption $\exists$ local chart $U \subseteq M$, s.t. $U=\Gamma_{f}$ where up to rigid motion and homothety $f: I^{d} \rightarrow \mathbb{R}^{n-d}$ is twice differentiable and $f(0)=0$. Since $\Delta_{*} T M \supseteq \Delta_{*} T U=\Phi_{f}\left(I^{d} \times \mathbb{R}^{d}\right)$, it suffices to prove that $\operatorname{int}\left(\Phi_{f}\left(I^{d} \times \mathbb{R}^{d}\right)\right) \neq \varnothing$.

    Since for any $b \in \mathbb{R}^{d}, f(x), J_{x} f$, and hence $\Phi_{f}(x, a)=\left(a, f(x)+\left(J_{x} f\right)(a-x)\right)$ are continuous and differentiable at $(0, b)$, and moreover $\left.\dfrac{\partial \Phi_{f}(x, a)}{\partial(x, a)} \right|_{(0,b)}=\left[\begin{array}{cc}
        0 & \left(d_{0}^{2} f\right) b \\
        I d_{d} & d_{0} f
    \end{array}\right]= 
    \left[\begin{array}{cc}
        0 & \left(b^{i} \dfrac{\partial f^{k}(0)}{\partial x^{i} \partial x^{j}}\right) \\
        I d_{d} & J_{0} f
    \end{array}\right]$, by Corollary \ref{coro1_22} it remians to proof that $\exists\ b_{0} \in \mathbb{R}^{d}$ s.t. $\left.\dfrac{\partial \Phi_{f}(x, a)}{\partial(x, a)}\right|_{\left(0, b_{0}\right)}$ has full rank, i.e. $\left(b_{0}^{i} \dfrac{\partial f^{k}(0)}{\partial x^{i} \partial x^{j}}\right)$ has full rank.

    Since by assumption $\left\{\left.\left[\begin{array}{cccc}\left(\dfrac{\partial f^{1}(0)}{\partial x^i \partial x^{j}}\right) & \cdots & \left(\dfrac{\partial f^{n-d}(0)}{\partial x^{i} \partial x^{j}}\right) & J_{0} f \\ 0 & \cdots & 0 & I d_{d}\end{array}\right] \xi \right| \xi \in \mathbb{R}^{(n-d+1) d}\right\}=T_{p}^{2} M=\mathbb{R}^{n}$, we have that $\left[\left(\dfrac{\partial f^{1}(0)}{\partial x^{i} \partial x^{j}}\right) \cdots\left(\dfrac{\partial f^{n-d}(0)}{\partial x^{i} \partial x^{j}}\right)\right]$ has full rank. Therefore $\exists\ b_{0}\in \mathbb{R}^{d}$ s.t. the contraction $\left(b_{0}^{i} \dfrac{\partial f^{k}(0)}{\partial x^{i} \partial x^{j}}\right)$ has full rank.

    This concludes the proof.
\end{proof}

\section{Further results for hypersurfaces}
In section \ref{sec1} we give a criterion for a $C^{1}$ submanifold $M$ to have $\operatorname{int}\left(\Delta_{*} T M\right) \neq \varnothing$. It turns out that if $M$ is a hypersurface then this result can be significantly improved.

\subsection{Tangent Bundle over Hypersurfaces}
We start with the following important observation:

\begin{lemma}\label{lemma3_1}
    If $f: I^{n} \rightarrow \mathbb{R}$ is a differentiable function, then $\forall\ a \in I^{n}$, $\phi_{f}^{a}\left(I^{n}\right)$ is path-connected.
\end{lemma}
\begin{proof}
    Take any $x \in I^{n} \backslash\{a\}$, denote $\lambda:=\dfrac{x+a}{2}$, then $(\lambda-a) t+\lambda, t \in[-1,1]$ is a segment connecting $a$ and $x$.

    Define $h(t):=\phi_{f}^{a}((\lambda-a) t+\lambda), t \in[-1,1]$, then $h(t)=f((\lambda-a) t+\lambda)+\nabla_{(\lambda-a) t+\lambda} f\left(a-((\lambda-a) t+\right.$ $\lambda)^{T}=f((\lambda-a) t+\lambda)-(t+1)\left(\nabla_{(\lambda-a)t+\lambda} f\right)(\lambda-a)$. Define $H(t)=2 \displaystyle\int_{0}^{t} f((\lambda-a) s+\lambda) d s-(t+1) f((\lambda-a) t+\lambda)$, $t \in[-1,1]$, then $\dot{H}(t)=2 f((\lambda-a) t+\lambda)-f((\lambda-a) t+\lambda)-(t+1)\dfrac{d\left(\left(\lambda^{i}-a^{i}\right) t+\lambda^{i}\right)}{d t} \dfrac{\partial f((\lambda-a) t+\lambda)}{\partial\left(\left(\lambda^{i}-a^{i}\right) t+\lambda^{i}\right)}=f((\lambda-a) t+\lambda)-(t+1)\left(\lambda^{i}-a^{i}\right)\dfrac{\partial f}{\partial x^{i}}((\lambda-a) t+\lambda)=f((\lambda-a) t+\lambda)-(t+1)\left(\nabla_{(\lambda-a) t+\lambda} f\right)(\lambda-a)=h(t)$. By Darboux theorem $h$ has the intermediate value property. Since $h(-1)=\phi_{f}^{a}(-(\lambda-a)+\lambda)=\phi_{f}^{a}(a)$ and $h(1)=\phi_{f}^{a}((\lambda-a)+\lambda)=\phi_{f}^{a}\left(2 \dfrac{x+a}{2}-a\right)=\phi_{f}^{a}(x)$, we obtain that $\phi_{f}^{a}\left(I^{n}\right) \supseteq \phi_{f}^{a}(\operatorname{conv}\{a, x\})=h([-1,1]) \supseteq \operatorname{conv}\left\{\phi_{f}^{a}(a), \phi_{f}^{a}(x)\right\}$.

    Take any $x_{1}, x_{2} \in I^{n} \backslash\{a\}$, since $\phi_{f}^{a}\left(I^{n}\right) \subseteq \mathbb{R}$, we have conv $\left\{\phi_{f}^{a}\left(x_{1}\right), \phi_{f}^{a}\left(x_{2}\right)\right\}\subseteq \operatorname{conv} \left\{\phi_{f}^{a}(a), \phi_{f}^{a}\left(x_{1}\right)\right\}$ $\cup \operatorname{conv}\left\{\phi_{f}^{a}(a), \phi_{f}^{a}\left(x_{2}\right)\right\}$. Since $\phi_{f}^{a}\left(I^{n}\right) \supseteq \operatorname{conv}\left\{\phi_{f}^{a}(a), \phi_{f}^{a}\left(x_{k}\right)\right\}, k=1,2$, we conclude that $\phi_{f}^{a}\left(I^{n}\right) \supseteq \operatorname{conv}\left\{\phi_{f}^{a}(a), \phi_{f}^{a}\left(x_{1}\right)\right\} \cup \operatorname{conv}\left\{\phi_{f}^{a}(a), \phi_{f}^{a}\left(x_{2}\right)\right\}\supseteq \operatorname{conv}\left\{\phi_{f}^{a}\left(x_{1}\right), \phi_{f}^{a}\left(x_{2}\right)\right\}$.

    This proves that $\phi_{f}^{a}$ is path-connected.
\end{proof}

Here we give a straightforward appliation of Lemma \ref{lemma3_1} to real functions.

\begin{corollary}\label{coro3_2}
    If $f: I \rightarrow \mathbb{R}$ is a nonlinear differentiable function, then for any $a \in I$, $\phi_{f}^{a}(I)$ is an interval with positive (or infinite) length.
\end{corollary}
\begin{proof}
    By Lemma \ref{lemma3_1}, $\phi_{f}^{a}(I)$ is a connected subset of $\mathbb{R}$, i.e. a singleton or an interval with positive (or infinite) length. Therefore it suffices to show that $\forall\ a \in \mathbb{R}, \phi_{f}^{a}$ is not constant.

    Provided that $\exists\ a \in \mathbb{R}$ s.t. $\phi_{f}^{a} \equiv b \in \mathbb{R}$, then we have $f(x)+(a-x) f^{\prime}(x)=\phi_{f}^{a}(x)=b$, $\forall\ x \in \mathbb{R}$. Solve the ordinary differentiable equation $y+(a-x) y^{\prime}=b$ by separating variables, we obtain $y+(a-x) \dfrac{d y}{d x}=b \Longleftrightarrow \dfrac{d y}{y-b}=\dfrac{d x}{x-a} \Longleftrightarrow \ln |y-b|=\ln| x-a|+C \Longleftrightarrow y-b=k(x-a)$, i.e. $y=k(x-a)+b$, where $C, k \in \mathbb{R}$ are constants. Therefore $\exists\ k \in \mathbb{R}$ s.t. $f(x)=k(x-a)+b$, contradicts the nonlinearity of $f$. Therefore $\forall\ a \in \mathbb{R}, \phi_{f}^{a}$ is not constant.
\end{proof}

Now we are prepared to state the main theorems.

\begin{proposition}\label{prop3_3}
    If $f: I^{n} \rightarrow \mathbb{R}$ is a nonlinear differentiable function, then $\operatorname{int}\left(\Delta_{*} T \Gamma_{f}\right) \neq \varnothing$.
\end{proposition}
\begin{proof}
    By Corollary \ref{coro1_14}, $\exists$ affine map $\ell: I \rightarrow I^{n}$ s.t. $f \circ \ell$ is nonlinear, moreover clearly $f \circ \ell$ is differentiable. Therefore by Corollary \ref{coro3_2} $\exists\ s \in \mathbb{R}$ s.t. $\phi_{f \circ \ell}^{s}(I)$ is not a singleton. Since by Lemma \ref{lemma1_19} $\phi_{f}^{\ell(s)}\left(I^{n}\right) \supseteq \phi_{f}^{\ell(s)}(\ell(I))=\phi_{f\circ\ell}^{s}(I)$ we have that $\phi_{f}^{\ell(s)}\left(I^{n}\right)$ is not a singleton.

    Denote $a_{0}:=\ell(s)$, then $\exists\ x_{1}, x_{2} \in I^{n}$, s.t. $\phi_{f}^{a_{0}}\left(x_{1}\right)<\phi_{f}^{a_{0}}\left(x_{2}\right)$. Since $\forall\ i=1,2, \phi_{f}^{a}\left(x_{i}\right)=f\left(x_{i}\right)+\left(J_{x_{i}} f\right)\left(a-x_{i}\right)$ is linear (and hence continous) in $a$, denote $\delta:=\dfrac{1}{2}\left(\phi_{f}^{a_{0}}\left(x_{2}\right)-\phi_{f}^{a_{0}}\left(x_{1}\right)\right)$ then $\exists\ \varepsilon>0$ s.t. $\forall\ a \in B_{\varepsilon}^{n}\left(a_{0}\right)$ we have $\left\|\phi_{f}^{a}\left(x_{i}\right)-\phi_{f}^{a_{0}}\left(x_{i}\right)\right\|<\delta$. Therefore denote $y_{1}:=\phi_{f}^{a_{0}}\left(x_{1}\right)+\delta$ and $y_{2}:=\phi_{f}^{a_{0}}\left(x_{2}\right)-\delta$ then $\forall\ a \in B_{\varepsilon}^{n}\left(a_{0}\right)$ we have $\phi_{f}^{a}\left(x_{1}\right)\leqslant y_{1}<y_{2} \leqslant \phi_{f}^{a}\left(x_{2}\right)$. By Lemma \ref{lemma3_1} $\forall\ a \in B_{\varepsilon}^{n}\left(a_{0}\right)$ we have $\phi_{f}^{a}\left(I^{n}\right) \supseteq\left.\right] y_{1}, y_{2}\left[\right.$. Therefore $\Delta_{*} T \Gamma_{f}=\Phi_{f}\left(I^{n} \times \mathbb{R}^{n}\right) \supseteq \Phi_{f}\left(I^{n} \times B_{\varepsilon}^{n}\left(a_{0}\right)\right)=\bigcup_{a \in B_{\varepsilon}^{n}\left(a_{0}\right)}\{a\} \times \varphi_{f}^{a}\left(I^{n}\right)\supseteq \bigcup_{a \in B_{\varepsilon}^{n}\left(a_{0}\right)}\{a\} \times\left.\right] y_{1}, y_{2}\left[\right.=B_{\varepsilon}^{n}\left(a_{0}\right) \times\left.\right] y_{1}, y_{2}\left[\right.$ where $B_{\varepsilon}^{n}\left(a_{0}\right) \times\left.\right] y_{1}, y_{2}\left[\right.$ is open in $\mathbb{R}^{n+1}$.

    This proves that $\operatorname{int}\left(\Delta_{*} T \Gamma_{f}\right) \neq \phi$.
\end{proof}

\begin{theorem}\label{theo3_4}
    If $M=M_{1} \times \cdots \times M_{m}$ is a product of nonlinear differentiable hypersurfaces, then $\operatorname{int}\left(\Delta_{*} T M\right) \neq \varnothing$.
\end{theorem}
\begin{proof}
    For any $i=1, \cdots, m$, by Corollary \ref{coro1_11}, $\exists$ local chart $U_{k} \subseteq M_{k}$ s.t. $U_{k}=\Gamma_{f_{k}}$ where up to rigid motion and homothety $f_{k}: I^{n_{k}} \rightarrow \mathbb{R}$ is nonlinear and differentiable. By Proposition \ref{prop3_3} $\exists$ open subset $W_{k}$ of $\mathbb{R}^{n_{k}+1}$ s.t. $\Delta_{*} T \Gamma_{f_{k}} \supseteq W_{k}$. Therefore $\Delta_{*} T M=\Delta_{*} T M_{1} \times \cdots \times \Delta_{*} T M_{m}\supseteq \Delta_{*} T U_{1} \times \cdots \times \Delta_{*} T U_{m}=\Delta_{*} T \Gamma_{f_{1}} \times \cdots \times \Delta_{*} T \Gamma_{f_{m}} \supseteq W_{1} \times \cdots \times W_{m}$, where $W_{1} \times \cdots \times W_{m}$ is an open subset of the ambient space.

    This proves that $\operatorname{int}\left(\Delta_{*} T M\right) \neq \varnothing$.
\end{proof}

\begin{corollary}\label{coro3_5}
    If $M$ is a nonlinear differentiable d-submanifold, then $\operatorname{dim}_{\mathcal{H}} \Delta_{*} T M \geqslant d+1$.
\end{corollary}
\begin{proof}
    By Corollary \ref{coro1_11} $\exists$ local chart $U \subseteq M$ s.t. $U=\Gamma_{f}$ where up to rigid motion and homothety $f: I^{d} \rightarrow \mathbb{R}^{n-d}$ is nonlinear and differentiable. It suffices to prove that $\operatorname{dim}_{\mathcal{H}} \Phi_{f}\left(I^{d} \times \mathbb{R}^{d}\right) \geqslant d+1$.

    Since $f=\left(f^{1}, \cdots, f^{n-d}\right)$ is nonlinear and differentiable, there exists $k \in\{1, \cdots, n-d\}$ s.t. $f^{k}: I^{d} \rightarrow \mathbb{R}$ is nonlinear and differentiable. By Proposition \ref{prop3_3} we have $\operatorname{dim}_{\mathcal{H}} \Phi_{f^{k}}\left(I^{d} \times \mathbb{R}^{d}\right)=d+1$. Denote by $\pi_{k}: \mathbb{R}^{n-d} \rightarrow \mathbb{R}$ the projection onto the $k$-th coordinate, then $\Phi_{f^{k}}\left(I^{d} \times \mathbb{R}^{d}\right)=\Phi_{\pi_{k}\circ f}\left(I^{d} \times \mathbb{R}^{d}\right)=\left(\pi_{k} \times I d_{d}\right) \circ \Phi_{f}\left(I^{d} \times \mathbb{R}^{d}\right)$. Since $\pi_{k} \times I d_{d}$ is Lipschitz, we have that $\operatorname{dim}_{\mathcal{H}} \Phi_{f}\left(I^{d} \times\right.$ $\left.\mathbb{R}^{d}\right) \geqslant \operatorname{dim}_{\mathcal{H}} \Phi_{f^{k}}\left(I^{d} \times \mathbb{R}^{d}\right)=d+1$.

    This concludes the proof.
\end{proof}

\subsection{Hyperplane covering of a hypersurface}
Moreover, if we assume more on the smoothness then we are able to obtain a much stronger result.

\begin{theorem}\label{theo3_6}
    Let $M$ be a nonlinear Lipschitz-continuously differentiable hypersurface in $\mathbb{R}^{n+1}$. If $\mathscr{A}$ is a family of hyperplanes in $\mathbb{R}^{n+1}$ s.t. $\cup\ \mathscr{A} \supseteq M$, then $\operatorname{dim}_{\mathcal{H}} \cup\ \mathscr{A}=n+1$.
\end{theorem}
\begin{proof}
    By Corollary \ref{coro1_11} $\exists$ local chart $U \subseteq M$ s.t. $U=\Gamma_{f}$ where $f: \Omega \longrightarrow \mathbb{R}$ is nonlinear and differentiable with Lipschitz-continuous gradient. Since $f$ is nonlinear and continuous, by Proposition \ref{prop1_16} up to rigid motion and homothety we have $I^{n} \subseteq \Omega$ and $\exists$ non-empty open subset $B$ of $I^{n-1}$, s.t. $\forall\ x_{0} \in B$ we have $f_{x_{0}}(t):=f\left(x_{0}^{1}, \cdots, x_{0}^{n-1}, t\right), t \in I$ is nonlinear. Since $\nabla_{(x, t)} f=\left(\dfrac{\partial f}{\partial x^{1}}, \cdots, \dfrac{\partial f}{\partial x^{n-1}}, \dfrac{\partial f}{\partial t}\right)$ is Lipschitz we have that $\dfrac{\partial f}{\partial t}$ is Lipschitz. Since $f_{x_{0}}^{\prime}(t)=\dfrac{\partial}{\partial t} f\left(x_{0}, t\right)$, we obtain that $f_{x_{0}}^{\prime}(t)$ is Lipschitz. Therefore $f_{x_{0}}^{\prime}(t)$ is absolutely continuous and hence satisfies Lusin's $N$-property.

    Take any $x_{0} \in B$, define $\Pi_{x_{0}}:=\left\{(x, t, y) \in \mathbb{R}^{n-1} \times \mathbb{R} \times \mathbb{R} \mid x=x_{0}\right\} \cong \mathbb{R}^{2}$ and $\mathscr{A}_{x_{0}}:=\left\{L \cap \Pi_{x_{0}} \mid \right.$ $\left.L\in \mathscr{A}\right\} \backslash\{\varnothing\}$, then $\Pi_{x_{0}} \supseteq\left\{\left(x_{0}, t, f_{x_{0}}(t)\right) \mid t \in I\right\}=\Gamma_{f_{x_{0}}}$ and $\cup\ \mathscr{A}_{x_{0}}=(\cup\ \mathscr{A}) \cap \Pi_{x_{0}}\supseteq(\cup\ \mathscr{A}) \cap \Gamma_{f_{x_{0}}}=\Gamma_{f_{x_{0}}}$. For any $L \cap \Pi_{x_{0}} \in \mathscr{A}_{x_{0}}$, since $\operatorname{codim} L=1, \operatorname{dim} \Pi_{x_{0}}=2$, and $L\cap \Pi_{x_{0}} \neq \varnothing$, by Krull's principle ideal theorem we have that $\operatorname{dim} L \cap \Pi_{x_{0}}=1$ or 2. Since $L$ and $\Pi_{x_{0}}$, and hence $L \cap \Pi_{x_{0}}$ are affine spaces, we have that $L \cap \Pi_{x_{0}}=\Pi_{x_0}$ or $L \cap \Pi_{x_{0}}$ is an affine line in $\Pi_{x_{0}}$. We claim that $\operatorname{dim}_{\mathcal{H}} \cup\ \mathscr{A}_{x_{0}}=2$. Indeed, if $\Pi_{x_{0}} \in \mathscr{A}_{x_{0}}$ then $\cup\ \mathscr{A}_{x_{0}}=\Pi_{x_{0}}$ and hence $\operatorname{dim}_{\mathcal{H}} \cup\ \mathscr{A}_{x_{0}}=2$, if $\Pi_{x_{0}} \notin \mathscr{A}_{x_0}$ then $\mathscr{A}_{x_{0}}$ is a family of affine lines covers $\Gamma_{f_{x_{0}}}$, by Corollary 9 in \cite{1} we have $\operatorname{dim}_{\mathcal{H}} \cup\ \mathscr{A}_{x_{0}}=2$.

    Define $\mathscr{A}_{B}:=\coprod\limits_{x_{0} \in B}\mathscr{A}_{x_{0}}$ then $\cup\ \mathscr{A}_B=\coprod\limits_{x_{0} \in B}\left(\cup\mathscr{A}_{x_{0}}\right)\subseteq \cup\ \mathscr{A} \text { and }\left\{\cup\ \mathscr{A}_{x_{0}}\right\}_{x_{0} \in B}$ are vertical sections of $\cup\ \mathscr{A}$. Therefore by the Fubini type theorem (see e.g. \cite{2}) we have that $\operatorname{dim}_{\mathcal{H}} \cup\ \mathscr{A}_{B} \geqslant \operatorname{dim}_{\mathcal{H}} B+2=(n-1)+2=n+1$. Since $\cup\ \mathscr{A}_{B} \subseteq \cup\ \mathscr{A} \subseteq \mathbb{R}^{n+1}$, we obtain that $\operatorname{dim}_{\mathcal{H}} \cup\ \mathscr{A}=n+1$.
\end{proof}

\begin{theorem}\label{theo3_7}
    Let $M$ be a nonlinear twice differentiable hypersurface in $\mathbb{R}^{n+1}$. If $A$ is a family of hyperplanes in $\mathbb{R}^{n+1}$ s.t. $\cup\ \mathscr{A} \supseteq M$, then $\operatorname{dim}_{\mathcal{H}} \cup\ \mathscr{A}=n+1$.
\end{theorem}
\begin{proof}
    Since $k$-th differentiability is preserved under restrictions and differentiable functions also satisfies Lusin's $N$-property, the same proof as in Theorem \ref{theo3_6} yields the desired result.
\end{proof}

The results above show that the union of any family of hyperplanes that covers a (sufficiently smooth) nonlinear hypersurface in $\mathbb{R}^{n+1}$ always has full Hausdorff dimension. However if we drop the nonlinearity condition then the Hausdorff dimension of the union can be any arbitrary real number between $n$ and $n+1$.

The corresponding constructions are based on the following theorem, which is proved by Falconer and Mattila using energy integrals.

\begin{theorem}
    (Theorem 3.2 in \cite{3}): Let $X, E \subseteq \mathbb{R}^{n+1}$ be non-empty Borel sets s.t. $\forall\ p \in E\ \mathcal{L}^{n}(X \cap L(p))>0$, then $\operatorname{dim}_{\mathcal{H}} X \cap \bigcup\limits_{p \in E} L(p)=\operatorname{dim}_{\mathcal{H}} \bigcup\limits_{p \in E} L(p)n+\min \left\{\operatorname{dim}_{\mathcal{H}} E, 1\right\}$.
\end{theorem}

Moreover, if $\operatorname{dim}_{\mathcal{H}} E>1$ then $\mathcal{L}^{n}\left(X \cap \bigcup\limits_{p \in E} L(p)\right)>0$. 

(Here for any $q:=(a, b) \in \mathbb{R}^{n} \times \mathbb{R}$ we denote by $L(q)$ the hyperplane $\left\{(x, y) \in \mathbb{R}^{n} \times \mathbb{R} \mid y=a^{T} x+b\right\}$)

\begin{proposition}\label{prop3_8}
    Let $M$ be a linear hypersurface in $\mathbb{R}^{n+1}$. For any $s \in[n, n+1]$, $\exists$ family $\mathscr{A}$ of hyperplanes in $\mathbb{R}^{n+1}$ s.t. $\cup\ \mathscr{A} \supseteq M$ and $\operatorname{dim}_{\mathcal{H}} \cup\ \mathscr{A}=s$.
\end{proposition}
\begin{proof}
    Wlog we assume that $M$ is an open subset of $L(0)$. Let $C$ be a $s-n$ dimensional Cantor set embedded in $\mathbb{R}^{n+1}$, and denote $X:=\mathbb{R}^{n+1}$ and $E:=C \cup\{0\}$. Define $\mathscr{A}=\{L(p) \mid p \in E\}$ then $\cup\ \mathscr{A} \supseteq L(0) \supseteq M$. Since $X, E \subseteq \mathbb{R}^{n+1}$ are Borel sets s.t. $\forall\ p \in E\ \mathcal{L}^{n}(X \cap L(p))=\mathcal{L}^{n}(L(p))=+\infty$, we obtain that $\operatorname{dim}_{\mathcal{H}} \cup\ \mathscr{A}=\operatorname{dim}_{\mathcal{H}} \bigcup\limits_{p \in E} L(p)=n+\min \left\{\operatorname{dim}_{\mathcal{H}} E, 1\right\}=n+(s-n)=s$.
\end{proof}

\section{Generic $C^1$ vector bundles}
\subsection{Genericity of Vector Bundles}
In this section we study the conditions for a $C^{1}$ vector bundle to attain expected size. In section \ref{sec2} we have shown that normal bundles always have large size while tangent bundles do not. Here, assume $C^{1}$ smoothness, we will prove that if a vector bundle of a submanifold $M$ is in general position with the tangent bundle of $M$ then it attains expected size.

Recall that a $k_{1}$-subspace $V_{1}$ and a $k_{2}$-subspace $V_{2}$ of $\mathbb{R}^{n}$ is in general position iff $\operatorname{dim} V_{1} \cap V_{2}=\max \left\{\left(k_{1}+k_{2}\right)-n, 0\right\}$. We extend this notion naturally to vector bundles, and give a definition of generic vector bundle in terms of tangency. 

\begin{definition}\label{defi4_1}
    A $C^1$ vector bundle $E$ defined by $\varphi: M \rightarrow \operatorname{Gr}_k(n)$ is said to be in general position with $TM$ at $p \in M$ iff $\operatorname{dim} \varphi(p) \cap T_{p} M=\max \{(k+d)-n, 0\}$. $E$ is said to be generic iff $\exists\ p \in M$ s.t. $E$ is in general position with $T M$ at $p$. 
    
    It is clear that in definition \ref{defi4_1} if $k \leqslant n-d$, then $E$ is in general position with $T M$ at $p \in M$ iff $\varphi(p) \cap T_{p} M=0$.
\end{definition}

\begin{proposition}\label{prop4_2}
    Let $E$ be a generic $C^{1}$ $k$-vector bundle over $C^{1}$ $d$-submanifold $M$ in $\mathbb{R}^{n}$ and $m:=\min \{k+d, n\}$, then $\mathcal{L}^{m}\left(\Delta_{*} E\right)>0$. Moreover, if $m=n$ then $\operatorname{int}\left(\Delta_{*} E\right) \neq \varnothing$.
\end{proposition}
\begin{proof}
    Denote by $\varphi: M \rightarrow \operatorname{Gr}_{k}(n)$ the $C^{1}$ map defines $E$. Wlog assume that $0 \in M$ and $E$ is in general position with $T M$ at 0. We consider the following two cases.

    Case 1: $k \leqslant n-d$. In this case $m=k+d$.

    Since $\varphi(0) \cap T_{0} M=0$, there exists a $(n-d)$-subspace $V$ of $\mathbb{R}^{n}$ s.t. $\psi(0) \subseteq V$ and $V$ is transverse to $T_0M$. By continuity $\exists$ common trivializing chart $U$ of $E$ and $T M$ at the vicinity of 0 , s.t. $\forall\ p \in U, T_{p} M$ is transverse to $V$. Denote by $P_{V}$ the orthogonal projection onto $V^{\bot}$, then $\left(U, P_{V}\right)$ is a chart of $M$, i.e. $\Omega:=P_{V}(U)$ is open in $V^{\perp} \equiv \mathbb{R}^{d}$ and $\exists\ C^{1}$ function $f: \Omega \rightarrow \mathbb{R}^{n-d}$ s.t. $U=\Gamma_{f}$. Let $F: \Omega \times \mathbb{R}^{k} \rightarrow \mathbb{R}^{2 n}$ be the trivialization over $U$, i.e. $F(x, \xi)=((x, f(x)),(S(x) \xi, T(x) \xi))$ where $S: \Omega \rightarrow \mathbb{R}^{k \times d}$ and $T: \Omega \rightarrow \mathbb{R}^{k\times (n-d)}$ are matrix-valued $C^{1}$ functions, then by the construction of $P_{V}$ we have $S(0)=0$ and $r k T(0)=k$. Therefore $\Delta_{*} F(x, \xi)=(x+S(x) \xi, f(x)+T(x) \xi)$ and hence $\left.\dfrac{\partial \Delta_{*} F}{\partial(x, \xi)}\right|_{(0,0)}=\left[\begin{array}{cc}I d_{d}+\left(\dfrac{\partial S_{i}^j(x)}{\partial x^{k}} \xi^{i}\right) & * \\ S(x) & T(x)\end{array}\right]_{(0,0)}=\left[\begin{array}{cc}I d_{d} & * \\ 0 & T(0)\end{array}\right]$ has full rank $k+d$. Therefore by the implicit function theorem, $\exists$ open subset $\Omega_{0}$ of $\Omega \times \mathbb{R}^{k}$, s.t. $\left.\Delta_{*} F\right|_{\Omega_{0}}$ parametrizes a $C^{1}(k+d)$-submanifold of $\mathbb{R}^{n}$ and hence $\mathcal{L}^{k+d}\left(\Delta_{*} F\left(\Omega_{0}\right)\right)={\rm{Vol}}_{k+d}\left(\Delta_{*} F\left(\Omega_{0}\right)\right)>0$. Since $\Delta_{*} E \supseteq \Delta_{*} F\left(\Omega \times \mathbb{R}^{k}\right) \supseteq \Delta_{*} F\left(\Omega_{0}\right)$, we obtain that $\mathcal{L}^{k+d}\left(\Delta_{*} E\right) \geqslant \mathcal{L}^{k+d}\left(\Delta_{*} F\left(\Omega_{0}\right)\right)>0 .$ Moreover, if $k+d=:m=n$ then $\Delta_{*} F\left(\Omega_{0}\right)$ is a $n$-submanifold, i.e. an open subset of $\mathbb{R}^{n}$, and hence $\operatorname{int} \left(\Delta_{*} E\right) \supseteq \Delta_{*} F\left(\Omega_{0}\right) \neq \varnothing$.

    Case 2: $k>n-d$. In this case $m=n$.

    Denote $\ell:=(k+d)-n$. Since $\operatorname{dim} \varphi(0) \cap T_{0} M=\ell$, by continuity $\exists$ common trivializing chart $N$ of $E$ and $T M$ at the vicinity of 0 s.t. $\forall\ p \in N$ we have $\operatorname{dim} \varphi(p) \cap T_{p} M \equiv \ell$. Therefore pullback $E$ by inclusion $N \stackrel{i}{\hookrightarrow} M$ then $E_{0}:=i^{*} E \cap T N$ is locally trivial and hence a $C^{1}$ $\ell$-vector bundle over $N$. By Swan's theorem we have the splitting $i^{*} E=E_{0} \oplus E_{0}^{\perp}$ where $E_{0}^{\perp}$ is a $C^{1}(k-l)$-vector bundle over $N$. Since $k-\ell=k-((k+d)-n)=n-d$ and $E_{0}^{\perp} \cap T N=N \times\{0\}$ we have that $E_{0}^{\perp}$ is a generic $C^{1}(n-d)$-vector bundle over $C^{1}$ d-submanifold $N$ in $\mathbb{R}^{n}$. By Case 1 we have $\operatorname{int}\left(\Delta_{*} E_{0}^{\perp}\right) \neq \varnothing$. Since $E \supseteq i^{*} E \supseteq E_{0}^{\perp}$ we conclude that $\Delta_{*} E \supseteq \Delta_{*} E_{0}^{\perp}$ has non-empty interior.
\end{proof}

Assume that the base manifold is not a point, then by Definition \ref{defi4_1} it is clear that a $C^{1}$ line bundle is non-generic iff it is a $C^{1}$ distribution, i.e. a $C^{1}$ line subbundle of the tangent bundle. In this case we also say that the bundle is tangent to the base manifold.

\begin{corollary}\label{coro4_3}
    Let $\mathscr{L}$ be a $C^{1}$ line bundle over $d$-submanifold $M$ of $\mathbb{R}^{n}$ (where $\left.1 \leqslant d \leqslant n-1\right)$. If $\mathscr{L}$ is not tangent to $M$ then $\mathcal{L}^{d+1}\left(\Delta_{*} \mathscr{L}\right)>0$. Moreover if $M$ is a hypersurface then $\operatorname{int}\left(\Delta_{*} \mathscr{L}\right) \neq \varnothing$.
\end{corollary}
\begin{proof}
    Restatement of Proposition \ref{prop4_2} for $k=1$.
\end{proof}

\subsection{Semi-generic Line Bundles}
We proceed by showing under which condition a non-generic $C^{1}$-vector bundle can attain its expected size.

Firstly, we give an example of a degenerated case:

Let $\varphi^{*} \tau_{k}(n)$ be a $C^{1}$ vector bundle over $C^{1}$ submanifold $M$ defined by $\varphi: M \rightarrow \operatorname{Gr}_{k}(n)$, then $\varphi_{1}:=\varphi\circ \varphi^{*} \pi_{1}$ in the following diagram pulls back a $C^{1}$ vector bundle over $C^{1}$ submanifold $\varphi^{*} \tau_{k}(n)$.
\begin{figure}[H]
    \centering 
    \begin{tikzpicture}
        \node at (0,0) {$\varphi^{*} \tau_{k}(n)$};
        \node at (4.2,0) {$\tau_{k}(n)$};
        \node at (0,3.3) {$\varphi_{1}^{*}\left(\varphi^{*} \tau_{k}(n)\right)$};
        \node at (0,-3.3) {$M$};
        \node at (4.2,-3.3) {$\operatorname{Gr}_{k}(n)$};
        \draw [dashed, thick, ->] (0,2.8)--(0,0.5);
        \node at (-0.7,1.6) {$\varphi_{1}^{*} \pi_{1}$};
        \draw [dashed, thick, ->] (0.7,2.8)--(3.9,0.5);
        \node at (2.4,2.1) {$\widetilde{\varphi}_{1}$};
        \draw [thick, ->] (0.9,0)--(3.5,0);
        \node at (2,0.3) {$\widetilde{\varphi}$};

        \draw [thick, ->] (0,-0.5)--(0,-2.8);
        \node at (-0.7,-1.6) {$\varphi^{*} \pi_{1}$};
        \draw [dashed, thick, ->] (0.7,-0.5)--(3.9,-2.8);
        \node at (2.4,-1.2) {$\varphi_{1}$};
        \draw [thick, ->] (0.9,-3.3)--(3.5,-3.3);
        \node at (2,-3) {$\varphi$};

        \draw [thick, ->] (4.2,-0.5)--(4.2,-2.8);
        \node at (4.7,-1.6) {$\pi_{1}$};
    \end{tikzpicture}
    \caption{}
    \label{fig3}
\end{figure}

Then we have $\varphi_{1}^{*}\left(\varphi^{*} \tau_{k}(n)\right)=\left\{((p, u),(0, v)) \in \mathbb{R}^{2 n} \times \mathbb{R}^{2 n} \mid u, v \in \varphi(p), p \in M\right\}$ and hence $\Delta_{*} \varphi_{1}{ }^{*}$ $\left(\varphi^{*} \tau_{k}(n)\right)=\{(p, u+v) \mid u, v \in \varphi(p), p \in M\}=\left\{(p, u) \mid u \in \varphi(p), p \in M\right\}=\varphi^{*} \tau_{k}(n)$.

More concretely, consider the line bundle $\mathscr{L}$ on a ruled surface $S$ in $\mathbb{R}^{3}$ generated naturally by the ruling, then we have $\Delta_{*} \mathscr{L}=S$.

Notice that in the example above, all points on the base manifold are degenerated in the sense that they are inflection points. We will show that this is not a coincidence.

For simplicity we restrict our consideration to line bundles. In this case we are able to introduce a notion of degeneracy in terms of foliation.

By Frobenius' theorem any $C^{1}$ line subbundle of the tangent bundle is integrable and hence defines a regular $1$-foliation with $C^{2}$ leaves. Conversely, any regular $1$-foliation $\mathcal{F}$ with $C^{2}$ leaves generates $C^{1}$ line distribution $T\mathcal{F}$.

The $C^{2}$ smoothness of integral submanifolds allow us to get more information about further degeneracy.

\begin{definition}\label{defi4_4}
    A $C^{1}$ line distribution $\mathscr{D}:=T\mathcal{F}$ over $C^{1}$ submanifold $M:=\cup\mathcal{F}$ is said to be osculate to $M$ at $p$ iff $T_{p}^{2} \gamma \subseteq T_{p} M$ where $\gamma \in \mathcal{F}$ is the leaf contains $p$. $\mathscr{D}$ is said to be osculate to $M$ iff it is osculate to $M$ at all $p \in M$.
\end{definition}

\begin{definition}\label{defi4_5}
    A $C^1$ line distribution over $C^{2}$ submanifold $M$ is said to be semi-generic iff it is not tangent or not osculate to $M$.
\end{definition}

\begin{remark}\label{remark4_2_3}
    The condition $T_{p}^{2} \gamma \subseteq T_{p} M$ in Definition \ref{defi4_4} is equivalent to the vanishing of the normal curvature $\kappa_{n}$ of $\gamma$ at $p$.
\end{remark}

This definition of semi-genericity can also be justified by considering the following notion of deformation:

\begin{definition}\label{defi4_6}
    For any $C^{1}$ global section $\sigma: M \rightarrow \mathscr{L}$ of a $C^{1}$ line bundle $\mathscr{L}$ over $C^{1}$ submanifold $M$ in $\mathbb{R}^{n}$, define $h_{\sigma}$ by the following diagram, where $M_{\sigma}:=\Delta_{*}(\sigma(M))$.
\end{definition}

\begin{figure}[H]
    \centering 
    \begin{tikzpicture}
        \node at (0,0) {$M_{\sigma}$};
        \node at (3,0) {$\mathbb{R}^{n}$};
        \node at (0,3) {$M$};
        \node at (3,3) {$\mathscr{L}$};
        \draw [dashed, thick, ->] (0,2.6)--(0,0.5);
        \node at (-0.5,1.6) {$h_{\sigma}$};
        \draw [thick, ->] (3,2.6)--(3,0.5);
        \node at (3.5,1.6) {$\Delta_{*}$};
        \draw [thick, ->] (0.5,3)--(2.5,3);
        \node at (1.5,3.3) {$\sigma$};
        \draw [thick, ->] (0.6,0)--(2.5,0);
        \node at (1.5,-0.3) {$i_{M_{\sigma}}$};
        \draw [thick] (0.6,0.2) arc (90:270:0.1);
    \end{tikzpicture}
    \caption{}
    \label{fig4}
\end{figure}

$h_{\sigma}$ is called a deformation of $M$ along $\mathscr{L}$ induced by $\sigma$. Moreover, $h_{\sigma}$ is said to be $\varepsilon$-small or a $\varepsilon$-deformation iff $\left\|\Delta_{*} \sigma-i_{M}\right\|_{\infty}<\varepsilon$. 

For example, $I d_{M}$ is the deformation induced by the zero section $\theta$, and $\Delta_{*} \theta=i_{M}$ implies that $I d_{M}$ is $\varepsilon$-small for all $\varepsilon>0$.

A deformation is by definition a $C^{1}$ surjection, but in general not necessarily invertible and its image is not necessarily a submanifold.

\begin{proposition}\label{prop4_7}
    Let $\mathscr{L}$ be a semi-generic $C^{1}$ line bundle over $C^1$ submanifold $M$, then $\forall\ \varepsilon>0 \exists$ $\varepsilon$-deformation $h_{\sigma}$ along $\mathscr{L}$ s.t. $M \stackrel{h_{\sigma}}{\cong} M_{\sigma}$ is a diffeomorphism of $C^{1}$ submanifolds and $\left(h_{\sigma}^{-1}\right)^{*} \mathscr{L}$ is generic over $M_{\sigma}$.
\end{proposition}
\begin{proof}
    If $\mathscr{L}$ is not tangent to $M$ then $I d_{M}$ is a required deformation. It remains to consider the case that $\mathscr{L}:=T \mathcal{F}$ is a distribution, where $\mathcal{F}$ is a regular $1$-foliation with $C^{2}$ leaves s.t. $\exists\ p_{0} \in \gamma \in \mathcal{F}$ s.t. $T_{p_0}^{2} \gamma \nsubseteq T_{p_0} M$.

    Wlog assume that $p_{0}=0 \in \mathbb{R}^{n}$. Let $U \subseteq M$ be a trivializing chart at the vicinity of 0 with chart map $\phi: U \rightarrow I^{d}$, and $C^{1}$ parametrization $\mathbf{f}: I^{d} \rightarrow \mathbb{R}^{n}$ s.t. $\mathbf{f}(0)=0$ and $\mathbf{f} \circ \phi=$ id. Since $\mathscr{L}$ is trivial on $U$, there exists a $C^{1}$ unit vector field $X: I^{d} \rightarrow \mathbb{S}^{n-1}$ s.t. $X \circ \phi$ generates $\mathscr{L}$ on $U$. Let $\rho: I^{d} \rightarrow \mathbb{R}$ be a compactly supported bump function with $\operatorname{supp}(\rho)=: K$ s.t. $\|\rho\|_{\infty}=\rho(0)=1$ and $\nabla_{0} \rho=0$. For any $t \in \mathbb{R}$, extend $t(\rho X) \circ \phi: U \rightarrow \mathbb{R}^{n}$ by zero to a $C^{1}$ vector field $v_{t}$ on $M$, and define $\sigma_{t}: M \rightarrow \mathscr{L}$ by $\sigma_{t}(p)=\left(p, v_{t}(p)\right)$, then $\sigma_{t}$ is a $C^{1}$ global section of $\mathscr{L}$. Denote by $h_{t}:=h_{\sigma_{t}}$ the deformation induced by $\sigma_{t}$ and $M_{t}:=M_{\sigma_{t}}$, then $\forall\ p \in U$ we have $h_{t}(p)=p+t(\rho X) \circ \phi(p)=(\mathbf{f}+t(\rho X)) \circ \phi(p)$, and $\forall\ p \in M \backslash \mathbf{f}(K) \supseteq M \backslash U$ we have $h_{t}(p)=p+0=p$. We claim that $\exists\ \delta>0$ s.t. $\forall\ t \in B_{\delta}^{1}$ $h_{t}$ is a $C^{1}$ diffeomorphism. It suffices to prove that $\exists\ \delta>0$ s.t. $\forall\ t \in B_{\delta}^{1}$, $h_{t} \circ \phi^{-1}=\mathbf{f}+t(\rho X): U \rightarrow \mathbb{R}^{n}$ is a $C^1$ diffeomorphism onto its image.

    For any $t \in \mathbb{R}$, define $\mathbf{f}_{t}:=\mathbf{f}+t(\rho X)$. Take any $x \in I^{d} \backslash K$, then $\forall\ t \in \mathbb{R}$, $J_{x} \mathbf{f}_{t}=J_{x} \mathbf{f}+t J_{x}(\rho X)=J_{x} \mathbf{f}+t \cdot 0=J_{x} \mathbf{f}$ has full rank. Take any $x \in K$, since $J_{x} \mathbf{f}$ has full rank, by continuity $\exists\ \delta_{x}>0$ s.t. $\forall\ t \in B_{\delta}^{1}$, $J_{x} \mathbf{f}_{t}$ has full rank. By the compactness of $K$, $\exists\ \delta>0$ s.t. $\forall\ x \in I^{d}, \forall\ t \in B_{\delta}^{1}$, $J_{x} \mathbf{f}_{t}$ has full rank. By the implicit function theorem $\forall\ t \in B_{\delta}^{1}$, we have that $\mathbf{f}_{t}$ parametrizes $C^{1}$ submanifold $U_{t}:=\mathbf{f}_{t}\left(I^{d}\right)$ and hence a $C^{1}$ diffeomorphism onto its image.

    This proves that $\forall\ t \in B_{\delta}^{1}$ we have $h_{t}: M \rightarrow M_{t}=(M \backslash U) \cup U_{t}$ is a diffeomorphism of $C^{1}$ submanifolds.

    For any $t \in B_{\delta}^1$, on $U_{t}=h_{t}(U),\left(h_{t}^{-1}\right)^{*} \mathscr{L}$ is generated by unit vector field $X\circ \phi\circ h_{t}^{-1}$, and we have $X \circ \phi \circ h_{t}^{-1}\left(h_{t}(0)\right)=X \circ \phi(0)=X(0)$ where $h_{t}(0)=\mathbf{f}_{t} \circ \phi(0)=\mathbf{f}_{t}(0)$. We claim that $\exists\ r \in] 0, \delta[$ s.t. $\forall\ t \in] 0, r[, \left(h_{t}^{-1}\right)^{*}\mathscr{L}$ is in general position with $M_{t}$ at $h_{t}(0)$, i.e. $X(0)$ is not tangent to $U_{t}$ at $\mathbf{f}_{t}(0)$.

    Parametrize $\gamma$ at the vicinity of 0 as a integral curve of unit vector field $X\circ \phi$ i.e. $\gamma(s):=\mathbf{f}(\alpha(s))$ where $\alpha: I \rightarrow I^{d}$ is a $C^{1}$ curve s.t. $\dot{\gamma}(s)=X(\alpha(s))$. Wlog assume that $\alpha(0)=0$ and denote $b:=\dot{\alpha}(0) \in \mathbb{R}^{n}$, then we have $X(0)=X(\alpha(0))=\dot{\gamma}(0)=\left.\dfrac{d}{d s}\right|_{s=0} \mathbf{f}(\alpha(s))=\left(J_{\alpha(0)} \mathbf{f}\right) \dot{\alpha}(0)=\left(J_{0} \mathbf{f}\right) b$ and $\ddot{\gamma}(0)=\left.\dfrac{d}{d s}\right|_{s=0} X(\alpha(s))=\left(J_{\alpha(0)} X\right) \dot{\alpha}(0)=\left(J_{0} X\right) b$. Since $T_{p}^{2} \gamma \nsubseteq T_{p} M$, i.e. $\left[J_{0} \mathbf{f} \mid \ddot{\gamma}(0)\right]=\left[J_{0} \mathbf{f} \mid\left(J_{0} X\right) b\right]$ has full rank, by continuity $\exists\ \delta^{\prime} \in ] 0, \delta[$, s.t. $\forall\ t \in ]0, \delta^{\prime}[$, $\left[J_{0} \mathbf{f}+t J_{0} X  \mid\left(J_{0} X\right) b\right]$ has full rank. Apply elementary transform we have $\left[J_{0} \mathbf{f}+t J_{0} X \mid\left(J_{0} X\right) b\right] \sim\left[J_{0} \mathbf{f}+t J_{0} X \left|-\dfrac{1}{t}\left(J_{0} \mathbf{f}\right) b\right]\right.\sim\left[J_{0} \mathbf{f}+t J_{0} X \mid\left(J_{0} \mathbf{f}\right) b\right]=\left[J_{0} \mathbf{f}+t J_{0} X \mid X(0)\right]$ has full rank. Since $T_{0} U_{t}=\operatorname{span}\left\{\dfrac{\partial \mathbf{f}_{t}}{\partial x^{i}}(0)\right\}$ and $J_{0} \mathbf{f}_{t}=J_{0} \mathbf{f}+t J_{0}(\rho X)=J_{0} \mathbf{f}+t \nabla \rho(0) \cdot X(0)+t \rho(0) J_{0} X=J_{0} \mathbf{f}+t J_{0} X$ we conclude that $X(0) \notin T_{0} U_{t}$, i.e. $X(0)$ is not tangent to $U_{t}$ at $\mathbf{f}_{t}(0)$. This proves the claim.

    Take $\lambda:=\dfrac{1}{2} \min \left\{\varepsilon, \delta, \delta^{\prime}\right\}$, then $h_{\lambda}$ is a deformation along $\mathscr{L}$ s.t. $M \stackrel{h_{\lambda}}{\cong} M_{\lambda}$ is a diffeomorphism of $C^{1}$ submanifolds and $\left(h_{\lambda}^{-1}\right)^{*} \mathscr{L}$ is generic over $M_{\lambda}$. Since $\left\|\Delta_{*} \sigma_{\lambda}-i_{M}\right\|_{\infty}=\left\|\mathbf{f}_{t}-\mathbf{f}\right\|_{\infty}=\|\lambda(\rho X)\|_{\infty} \leqslant|\lambda| \cdot 1 \leqslant \dfrac{\varepsilon}{2}<\varepsilon$, we have that $h_{\lambda}$ is $\varepsilon$-small.

    This concludes the proof.
\end{proof} 

Now we are prepared to state the main theorem. It follows straightforwardly from the following lemma which claim that deformation along a line bundle preserves codiagonal.

\begin{lemma}\label{lemma4_8}
    Let $\sigma$ be a $C^{1}$ line bundle $\mathscr{L}$ over $C^{1}$ submanifold $M$ of $\mathbb{R}^{n}$ if $h_{\sigma}$ is a $C^{1}$ diffeomorphism then $\Delta_{*}\left(h_{\sigma}^{-1}\right)^{*} \mathscr{L}=\Delta_{*} \mathscr{L}$.
\end{lemma}
\begin{proof}
    Denote by $\varphi: M \rightarrow \mathbb{P}^{n-1}$ the $C^{1}$ map defines $\mathscr{L}$. Since $\forall\ p \in M, \sigma(p)\in \varphi(p)$, we have that $\Delta_{*}\left(h_{\sigma}^{-1}\right)^{*} \mathscr{L}=\Delta_{*}\left\{\left(h_{\sigma}(p), u\right) \mid u \in \varphi(p), p \in M\right\}=\{p+\sigma(p)+u \mid u \in \varphi(p), p \in M\}=\{p+u \mid u \in \varphi(p), p \in M\}=\Delta_{*}\{(p, u) \mid u \in \varphi(p), p \in M\}=\Delta_{*} \mathscr{L}$.
\end{proof}

\begin{theorem}\label{theo4_9}
    Let $\mathscr{L}$ be a semi-generic $C^{1}$ line bundle over a $C^{1}$ $d$-submanifold $M$ in $\mathbb{R}^{n}$ (where $\left.1 \leqslant d \leqslant n-1\right)$ then $\mathcal{L}^{d+1}\left(\Delta_{*} \mathscr{L}\right)>0$. Moreover, if $M$ is a hypersurface then $\operatorname{int}\left(\Delta_{*} \mathscr{L}\right) \neq \varnothing$.
\end{theorem}
\begin{proof}
    Since $\mathscr{L}$ is semi-generic, by Proposition \ref{prop4_7} $\exists$ deformation $h_{\sigma}$ along $\mathscr{L}$ s.t. $h_{\sigma}$ is a $C^{1}$ diffeomorphism and $\left(h_{\sigma}^{-1}\right)^{*} \mathscr{L}$ is generic, By Corollary \ref{coro4_3} we have $\mathcal{L}^{d+1}\left(\Delta_{*}\left(h_{\sigma}^{-1}\right)^{*} \mathscr{L}\right)>0$, and if $d=n-1$ then $\operatorname{int}\left(\Delta_{*}\left(h_{\sigma}^{-1}\right)^{*} \mathscr{L}\right) \neq \varnothing$. By Lemma \ref{lemma4_8} we have $\Delta_{*} \mathscr{L}=\Delta_{*}\left(h_{\sigma}^{-1}\right)^{*} \mathscr{L}$, and the statement follows immediately.
\end{proof}

We discuss an application of Theorem \ref{theo4_9} to $C^{2}$ submanifolds. In this case the degeneracy of line bundles can be restricted by the geometry of the base manifold.

\begin{lemma}\label{lemma4_10}
    Let $M$ be a $C^{2}$ submanifold, $\gamma$ be a twice differentiable regular curve on $M$, and $p$ be a point on $\gamma$. If $T_{p}^{2} \gamma \subseteq T_{p} M$ then $p$ is an inflection point of $M$.
\end{lemma}
\begin{proof}
    Take local chart $U \subseteq M$ at the vicinity of $p$ s.t. $U=\Gamma_{f}$ where up to rigid motion and homothety $f: I^{d} \rightarrow \mathbb{R}^{n-d}$ is a $C^{2}$ map with $(0, f(0))=p$ and $J_{0} f=0$.

    Parametrize $\gamma$ at the vicinity of $p$ by $\gamma(t):=\left(\alpha(t), f(\alpha(t))\right.$ where $\alpha: I \rightarrow I^{d}$ and $\ddot{\alpha}(0)=u$, then $\dot{\gamma}(t)=\left(\dot{\alpha}(t),\left(J_{\alpha(t)} f\right) \dot{\alpha}(t)\right)$ and $\ddot{\gamma}(0)=\left(\ddot{\alpha}(0),(J_{\alpha(0)} f) \ddot{\alpha}(0)+\dot{\alpha}(0)^{T}\left(\text{Hess}_{\alpha(0)} f\right) \dot{\alpha}(0)\right)=\left(u,\left(J_{0} f\right) u+v^{T}\left(\text{Hess}_0f\right) v\right)$, where $\text{Hess}_0f$ is the Hessian tensor of $f$ at $0$.

    Since $T_{p}^{2} \gamma \subseteq T_{p} M$, we have $\ddot{\gamma}(0) \in T_{p} M$. Since $\left(u,\left(J_{0} f\right) u\right) \in T_{p} M$ we have $\left(0, v^{T}\left(\text{Hess}_0f\right) v\right)=\ddot{\gamma}(0)-\left(u,\left(J_{0} f\right) u\right) \in T_{p} M$. Therefore we obtain that $\left(0, v^{T}\left(\text{Hess}_{0} f\right) v\right)=\left(0,\left(J_{0} f\right) 0\right)=(0,0)$, i.e. $v^{T}\left(\text{Hess}_0f\right)v=0$. Since $\gamma$ is regular, we have $(\dot{\alpha}(0), 0)=\left(\dot{\alpha}(0),\left(J_{0} f\right) \dot{\alpha}(0)\right)=\dot{\gamma}(0) \neq 0$, i.e. $v=\dot{\alpha}(0) \neq 0$. Therefore $\text{Hess}_{0} f$ is degenerated and $f$ is not a Morse function at 0 , i.e. $p=(0, f(0))$ is an inflection point of $M$.
\end{proof}

\begin{theorem}\label{theo4_11}
    Let $M$ be a $C^{2}$ $d$-submanifold of $\mathbb{R}^{n}$ (where $1 \leqslant d \leqslant n-1$). If not all points of $M$ are inflection points, then $\forall\ C^{1}$ line bundle $\mathscr{L}$ over $M$ we have $\mathcal{L}^{d+1}\left(\Delta_{*} \mathscr{L}\right)>0$, and moreover if $M$ is a hypersurface then $\operatorname{int}\left(\Delta_{*} \mathscr{L}\right) \neq \varnothing$.
\end{theorem}
\begin{proof}
    Since $\exists\ p \in M$ s.t. $p$ is not an inflection point, by Lemma \ref{lemma4_10} $\forall$ regular $1$-foliation $\mathcal{F}$ of $M$ with $C^{2}$ leaves, we have $T_{p}^{2} \gamma \nsubseteq T_{p} M$ where $\gamma$ is the leaf contains $p$. Therefore $\forall\ C^{1}$ line bundle $\mathscr{L}$ on $M, \mathscr{L}$ is semi-generic. The statement then follows immediately from Theorem \ref{theo4_9}.
\end{proof}

\section{Line bundles over hyperplanes and ellipsoids}
In this section we focus on (continuous) line bundles over classical geometric objects such as hyperplanes, convex curves and ellipsoids.

\subsection{Algebraic Topological Lemmas}
Before we discuss the size of line bundles, we shall introduce several useful lemmas.

Recall that for any $u, v \in \mathbb{S}^{n} \subseteq \mathbb{R}^{n+1}$ define $\angle(u, v):=\arccos \left(u^{T} v\right)$ then $\angle$ coincides with the geodesic distance on $\mathbb{S}^{n}$ w.r.t. the induced Riemannian metric from $\mathbb{R}^{n+1}$.

\begin{definition}\label{defi5_1}
    For any $a \in \mathbb{S}^{n}$ and $\theta>0$, define $V_{\theta}(a):=\left\{u \in\mathbb{S}^{n} \mid \angle(u, a) \leqslant \theta\right\}$ the geodesic disk at $a$ with radius $\theta$, and $W_{\theta}(a):=\left\{u\in \mathbb{S}^{n} \mid \angle(u, a)=\theta\right\}$ its boundary.
\end{definition}

We have the following analog of Borsuk's non-retraction theorem. The proof is adapted from Lemma 2.5 in \cite{8} and the condition is modified and slightly weakened for our purpose.

\begin{lemma}\label{lemma5_2}
    Let $a \in \mathbb{S}^{n}$ and $0<\varepsilon<\theta<\dfrac{\pi}{2}$. If $f: V_{\theta}(a) \rightarrow \mathbb{S}^{n}$ is a continuous map s.t. $\exists\ b \in V_{\pi-\theta-2 \varepsilon}(-a)$ s.t. $b \notin f\left(V_{\theta}(a)\right)$, and $\forall\ u \in W_{\theta}(a)\angle(u, f(u))<\varepsilon$, then $f\left(V_{\theta}(a)\right) \supseteq V_{\theta-\varepsilon}(a)$.
\end{lemma}
\begin{proof}
    Provided that $\exists\ c \in V_{\theta-\varepsilon}(a)$ s.t. $c \notin f\left(V_{\theta}(a)\right)$. Denote $X:=\mathbb{S}^{n-1} \backslash\{b, c\}$ and $A:=V_{\theta+\varepsilon}(a) \backslash V_{\theta-\varepsilon}(a)$, i.e. the $\varepsilon$-neighborhood of $W_{\theta}(a)$. Since $\forall\ u \in A$ clearly $\exists !$ distance-minimizing geodesic connects $u$ and $W_{\theta}(a)$, we conclude that the metric projection $P:=P_{W_{\theta}(a)}: A \rightarrow W_{\theta}(a)$ is a well-defined (single valued) map. Since by assumption $f\left(W_{\theta}(a)\right) \subseteq A$, we have that $\left.P \circ f\right|_{W_{\theta}(a)}$ is a well-defined continuous map.

    Take any $u \in W_{\theta}(a)$, since $f(u) \in A$ the $\varepsilon$-neighborhood of $W_{\theta}(a)$, we have $\angle(f(u), P(f(u)))\leqslant\varepsilon$. Therefore $\angle(u, P(f(u))) \leqslant \angle(u, f(u))+\angle(f(u), P(f(u)))<\varepsilon+\varepsilon<2 \theta$. Denote by $u^{\prime}$ the antipodal point of $u$ in $W_{\theta}(a) \equiv \mathbb{S}^{n-1}$, then $\angle\left(u, u^{\prime}\right)=2 \theta$. Therefore $\forall u \in W_{\theta}(a)$ we have $P(f(u)) \neq u^{\prime}$. Therefore $\left.P\circ f\right|_{W_{\theta}(a)}$ is homotopic to identity. By Brouwer's theorem $\left.P \circ f\right|_{W_{\theta}(a)} \simeq id_{\mathbb{S}^{n-1}}$ is not null-homotopic, this gives an obstruction for extending $\left.P\circ f\right|_{W_{\theta}(a)}$ to $V_{\theta}(a)$. 
    
    Since $X \simeq W_{\theta}(a)$, it is easy to check that $P$ extends continuously to a deformation retraction $\tilde{P}: X \rightarrow W_{\theta}(a)$. Since by assumption $b,c\notin f\left(V_{\theta}(a)\right)$, i.e. $f\left(V_{\theta}(a)\right)\supseteq X$, we obtain that $\tilde{P}\circ f$ is a well-defined continuous map extends $\left.P \circ f\right|_{W_{\theta}(a)}$, contradiction.

    This proves that $f\left(V_{\theta}(a)\right) \supseteq V_{\theta-\varepsilon}(a)$.
\end{proof}

The next lemma describes the size of a special continuous family of half-lines parametrized by a disk.

\begin{proposition}\label{prop5_3}
    Take continuous maps $\alpha, \beta: \mathbb{D}^{n} \rightarrow \mathbb{R}^{n}$ s.t. $\left.\alpha\right|_{\partial \mathbb{D}^{n}}=I d$, and for any $x \in \mathbb{D}^{n}, \gamma_{x}(t):=(t, t \alpha(x)+\beta(x)), t \in \mathbb{R}_{+}$, then $\bigcup\limits_{x \in \mathbb{D}^{n}} \gamma_{x}$ contains an unbounded open subset of $\mathbb{R}^{n+1}$.
\end{proposition}
\begin{proof}
    Denote $\mathbb{S}_{+}^{n}:=\left\{(h, x) \in \mathbb{R} \times \mathbb{R}^{n} \mid\|(h, x)\|=1, h>0\right\}, \Pi:=\left\{(h, x) \in \mathbb{R} \times \mathbb{R}^{n} \mid h=1\right\}$, $D:=\{(1, x) \mid\|x\| \leqslant 1\}$ and $a:=(1,0) \in D \subseteq \Pi$, Define $\alpha_{1}: D \rightarrow \Pi$ by $\alpha_{1}(1, x):=(1, \alpha(x)), \beta_{0}: D \rightarrow \mathbb{R}^{n+1}$ by $\beta_{0}(1, x)=(0, \beta(x))$, and $\pi: \Pi \longrightarrow \mathbb{S}_{+}^{n}$ by $\pi(p):=\dfrac{p}{\|p\|}$, then $\pi$ is a homeomorphism s.t. $\pi(D)=V_{\frac{\pi}{4}}(a) \subseteq \mathbb{S}_{+}^{n}$ and induces $\alpha_{1}^{*}:=\pi \circ \alpha_{1} \circ \pi^{-1}: V_{\frac{\pi}{4}}(a) \longrightarrow \mathbb{S}_{+}^{n}$. Apply the definitions above we rewrite $\gamma_{x}(t)=t \alpha_{1}(1, x)+\beta_{0}(1, x)=t \pi \circ \alpha_{1}(1, x)+\beta_{0}(1, x)=t\alpha_{1}^{*}(\pi(1, x))+\beta_{0} \pi^{-1}(\pi(1, x))$. For any $u \in V_{\frac{\pi}{4}}(a)$ denote $\gamma_{u}(t):=t \alpha_{1}^{*}(u)+\beta_{0} \pi^{-1}(u), t \in \mathbb{R}_{+}$, then $\left\{\gamma_{x}\right\}_{x \in \mathbb{D}^{n}}$ is in $1-1$ correspondence with $\left\{\gamma_{u}\right\}_{u \in V_{\frac{\pi}{4}}(a)}$, and hence $\bigcup\limits_{x \in \mathbb{D}^{n}} \gamma_{x}=\bigcup\limits_{u \in V_{\frac{\pi}{4}}(a)} \gamma_{u}$.

    Since $\beta_{0} \pi^{-1}\left(V_{\frac{\pi}{4}}(a)\right)=\beta_{0}(D)=\beta\left(\mathbb{D}^{n}\right) \times\{0\}$ is compact, there exists $r>0$ s.t. $B_{r}^{n+1}\supseteq \beta_{0} \pi^{-1}\left(V_{\frac{\pi}{4}}(a)\right)$. Take $\rho \geqslant r$ then by Jordan's curve theorem $\forall\ u \in V_{\frac{\pi}{4}}(a)$, $\gamma_{u} \pitchfork \partial B_{\rho}^{n+1}$ at an unique point $p_{u, \rho}$. Define $\psi_{\rho}: V_{\frac{\pi}{4}}(a) \longrightarrow \mathbb{S}^{n}$ by $\psi_{\rho}(u):=\dfrac{p_{u, \rho}}{\left\|p_{u, \rho}\right\|}=\dfrac{1}{\rho} p_{u, \rho}$ then it is easy to check that $\psi_{\rho}$ is continuous.

    By the construction of $\psi_{\rho}, \forall\ \varepsilon>0, \exists\ R \geqslant r$ s.t. $\forall\ \rho>R$ for any $u \in V_{\frac{\pi}{4}}(a)$ we have $\angle\left(\alpha_{1}^{*}(u), \psi_{\rho}(u)\right)<\varepsilon$. Take $\varepsilon=\dfrac{\pi}{12}$, then for any $\rho>R$ we have $\psi_{\rho}\left(V_{\frac{\pi}{4}}(a)\right) \subseteq \bigcup\limits_{u \in \mathbb{S}_{+}^{n}} V_{\frac{\pi}{12}}(u)=V_{\frac{7 \pi}{12}}(a)$, in particular we have $-a\notin \psi_{\rho}\left(V_{\frac{\pi}{4}}(a)\right)$. Since $\forall\ x \in \partial \mathbb{D}^{n}, \alpha_{1}(1, x)=(1, \alpha(x))=(1, x)$ i.e. $(1,x)$ is a fixed point of $\alpha_1$, and $\alpha_{1}^{*}$ is topological conjugate with $\alpha_1$, we conclude that $\forall\ u \in W_{\frac{\pi}{4}}(a)$, $\alpha_{1}^{*}(u)=u$. Therefore for any $\rho>R, \forall\ u \in W_{\frac{\pi}{4}}(a)$ we have $\angle\left(u, \psi_{\rho}(u)\right)=\angle\left(\alpha_{1}^{*}(u), \psi_{\rho}(u)\right)<\dfrac{\pi}{12}$. By Lemma \ref{lemma5_2} we have that $\psi_{\rho}\left(V_{\frac{\pi}{4}}(a)\right) \supseteq V_{\frac{\pi}{6}}(a)$. Therefore $\bigcup\limits_{x \in \mathbb{D}^{n}} \gamma_{x}=\bigcup\limits_{u \in V_{\frac{\pi}{4}}(u)} \gamma_{u}\supseteq \bigcup\limits_{\rho>R} \psi_{\rho}\left(V_{\frac{\pi}{4}}(a)\right)\supseteq \bigcup\limits_{\rho>R} \rho V_{\frac{\pi}{6}}(a) \supseteq\left\{t u \mid \angle(u, a)<\frac{\pi}{6}, t>0\right\}\backslash \bar{B}_{R}$, where $\left\{t u \mid \angle(u, a)<\dfrac{\pi}{6}, t>0\right\} \backslash \bar{B}_{R}$ is clearly an unbounded open subset of $\mathbb{R}^{n+1}$.
\end{proof}

By considering the family of lines in projective spaces and passing to an another affine chart, we obtain the following dual proposition.

\begin{proposition}\label{prop5_4}
    Take continuous map $\beta: \mathbb{D}^{n} \rightarrow \mathbb{R}^{n}, f \in C\left(\mathbb{D}^{n}\right)$ s.t. $\left.f \right|_{\partial \mathbb{D}^{n}} \equiv\text{constant}$, and for any $x \in \mathbb{D}^{n}, \ell_{x}(t):=(t+f(x), t \beta(x)+x), t \in \mathbb{R}_{+}$, then $\operatorname{int}\left(\bigcup\limits_{x \in \mathbb{D}^{n}} \ell_{x}\right) \neq \varnothing$.
\end{proposition}
\begin{proof}
    Wlog we assume that $\left.f\right|_{\partial \mathbb{D}^{n}} \equiv 0$.
    
    Identify $\mathbb{R}^{n+1}=\left\{\left(a^{1}, \cdots, a^{n+1}\right) \mid a^{i} \in \mathbb{R}\right\} \equiv\left\{\left[1: a^{1}: \cdots: a^{n+1}\right] \mid a^{i} \in \mathbb{R}\right\}=U_{0} \hookrightarrow \mathbb{P}^{n+1}$ with the standard affine chart and denote $U_{1}:=\left\{\left[a^{0}: 1 : a^{2}: \cdots: a^{n+1}\right] \mid a^{i} \in \mathbb{R}\right\}$.

    Denote $\Pi_{0}:=\mathbb{P}^{n+1} \backslash U_{1}=\left\{\left[1: 0: a^{2}: \cdots: a^{n+1}\right] \mid a^{i} \in \mathbb{R}\right\} \subseteq U_{0}$ and $\Pi_{1}:=\mathbb{P}^{n+1} \backslash U_{0}=\left\{\left[0: 1:\right.\right.$ $\left.\left.a^{2}:\cdots: a^{n+1}\right]\mid a^{i} \in \mathbb{R}\right\}\subseteq U_{1}$. In homogeneous coordinates we have $\beta\left(\left[1: 0: x^{1}: \cdots: x^{n}\right]\right)=\left[0: 1: \beta^{1}(x): \cdots: \beta^{n}(x)\right]$ and $\Gamma_{f}=\left\{\left[1: f(x): x^{1}: \cdots: x^{n}\right] \mid x \in \mathbb{D}^{n}\right\}$. For any $x \in \mathbb{D}^{n}$, define $\bar{\ell}_{x}: \mathbb{P}^{1} \longrightarrow \mathbb{P}^{n+1}$ by $\bar{\ell}_{x}([\lambda: \mu]):=\left[\lambda: \lambda f(x)+\mu: \lambda x^{1}+\mu \beta^{1}(x): \cdots : \lambda x^{n}+\mu \beta^{n}(x)\right]$, then $\bar{\ell}_{x} \cong \mathbb{P}^{1}$ is a projective line in $\mathbb{P}^{n+1}$.

    Solve $\lambda f(x)+\mu=1$ we obtain $\mu=1-\lambda f(x)$. Therefore $\bar{\ell}_{x} \cap U_{1}=\left\{\left[\lambda: 1: \lambda x^{1}+(1-\lambda f(x)) \beta^{1}(x)\right.\right.$ $\left.\left.:\cdots: \lambda x^{n}+(1-\lambda f(x)) \beta^{n}(x)\right]\mid \lambda \in \mathbb{R}\right\}=\left\{\left[\lambda: \ 1: \lambda\left(x^{1}-f(x) \beta^{1}(x)\right)+\beta^{1}(x): \cdots: \lambda\left(x^{n}-f(x)\right.\right.\right.$\\ $\left.\left.\left. \beta^{n}(x)\right)+\beta^{n}(x)\right]\mid \lambda \in \mathbb{R}\right\}\equiv\{(\lambda, \lambda(x-f(x) \beta(x))+\beta(x)) \mid \lambda \in \mathbb{R}\}$. Denote $\alpha(x):=x-f(x) \beta(x)$ then $\alpha: \mathbb{D}^{n} \rightarrow \mathbb{R}^{n}$ is a continuous map s.t. $\forall\ x \in \partial \mathbb{D}^{n}, \alpha(x)=x-0 \cdot \beta(x)=x$. Therefore by Proposition \ref{prop5_3} we have $\left(\bigcup\limits_{x \in \mathbb{D}^{n}} \bar{\ell}_{x}\right) \cap U_{1}$ has non empty interior in $\mathbb{R}^{n+1} \equiv U_{1}$ and hence has non empty interior in $\mathbb{P}^{n+1}$. In particular $\bigcup\limits_{x \in \mathbb{D}^{n}} \bar{\ell}_{x}$ has non empty interior in $\mathbb{P}^{n+1}$.

    Since $U_{0}$ is a dense open subset of $\mathbb{P}^{n+1}$, we have that $\left(\bigcup\limits_{x \in \mathbb{D}^{n}} \bar{\ell}_{x}\right) \cap U_{0}$ has non empty interior in $U_{0} \equiv \mathbb{R}^{n+1}$. Since $\bar{\ell}_{x} \cap U_{0}=\left\{\left[1:f(x)+\mu: x^{1}+\mu \beta^{1}(x): \cdots: x^{n}+\mu \beta^{n}(x)\right]\mid \mu \in \mathbb{R}\right\}\equiv\{(\mu+f(x), \mu \beta(x)+x) \mid \mu \in \mathbb{R}\}=\ell_{x}$. We conclude that $\bigcup\limits_{x \in \mathbb{D}^{n}} \ell_{x}=\bigcup\limits_{x \in \mathbb{D}^{n}}\left(\bar{\ell}_{x} \cap U_{0}\right)=\left(\bigcup\limits_{x \in \mathbb{D}^{n}} \overline{\ell}_{x}\right) \cap U_{0}$ has non empty interior in $\mathbb{R}^{n+1}$.
\end{proof}

In order to reformulate the propositions above using the language of fiber bundles, we introduce the following notion of transversality.

\begin{definition}\label{defi5_5}
    Let $E$ be a $k$-vector bundle over a subset $X$ of $\mathbb{R}^{n}$ defined by $\varphi: X \rightarrow\operatorname{Gr}_{k}(n)$, and $V$ be a $(n-k)$-subspace of $\mathbb{R}^{n}$, then $E$ is said to be transverse to $V$ (denote by $E \pitchfork V)$ iff $\forall\ x \in X\ \varphi(x) \pitchfork V$.
\end{definition}

\begin{theorem}\label{theo5_6}
    Let $\mathscr{L}$ be a line bundle over $\Gamma_{f}$ where $f \in C\left(\mathbb{D}^{n}\right)$ with $\left.f\right|_{\partial \mathbb{D}^{n}} \equiv\text{constant}$, if $\mathscr{L} \pitchfork \mathbb{R}^{n} \times\{0\}$, then $\operatorname{int}\left(\Delta_{*} \mathscr{L}\right) \neq \varnothing$.
\end{theorem}
\begin{proof}
    Coordinate-free restatement of Proposition \ref{prop5_4} using Definition \ref{defi5_5}.
\end{proof}

For simplicity, for any subset $X$ of $\mathbb{R}^{n}$ and $f \in C(X)$ we introduce notation $P_{f}:=i d_{X} \times 0: \Gamma_{f} \longrightarrow X \times\{0\}$ for the projection $(x, f(x)) \longmapsto(x, 0)$. Since it is clear that $P_{f}$ is a homeomorphism, pullback by $P_{f}^{*}$ is a $1-1$ correspondence, i.e. all line bundles over $\Gamma_{f}$ have the form $P^{*}_f \mathscr{L}$ where $\mathscr{L}$ is a line bundle over $X \times\{0\}$, and vice versa.

Therefore the following reformulation is also useful:
\begin{corollary}\label{coro5_7}
    Let $\mathscr{E}$ be a line bundle over $\mathbb{D}^{n} \times\{0\} \subseteq \mathbb{R}^{n} \times \mathbb{R}$ s.t. $\mathscr{E} \pitchfork \mathbb{R}^{n} \times\{0\}$ then $\forall\ f \in C\left(\mathbb{D}^{n}\right)$ with $\left.f\right|_{\partial \mathbb{D}^{n}} \equiv\text{constant}$, we have $\operatorname{int}\left(\Delta_{*} P_{f}^{*} \mathscr{E}\right) \neq \varnothing$.
\end{corollary}

\subsection{Line Bundle over Hyperplanes, Convex Curves, and Spheres}
Now we shall discuss the application of the theorems above.

We start with the line bundles over linear hypersurfaces and, in particular, hyperplanes.
\begin{theorem}\label{theo5_8}
    Let $M$ be a linear hypersurface and $\mathscr{L}$ be a line bundle over $M$, then $\operatorname{int}\left(\Delta_{*} \mathscr{L}\right) \neq \varnothing$ iff $\mathscr{L}$ is not tangent to $M$.
\end{theorem}
\begin{proof}
    Wlog assume that $M:=\Omega \times\{0\} \subseteq \mathbb{R}^{n} \times \mathbb{R}$ where $\Omega$ is an open subset of $\mathbb{R}^{n}$.

    $\Longrightarrow$: Provided that $\mathscr{L}$ is tangent to $M$, then $\Delta_{*} \mathscr{L} \subseteq \Delta_{*} T M=\mathbb{R}^{n} \times\{0\}$ and hence $\operatorname{dim}_{\mathcal{H}} \Delta_{*} \mathscr{L} \leqslant \operatorname{dim} \mathbb{R}^{n}=n$, in particular $\operatorname{int}\left(\Delta_{*} \mathscr{L}\right)=\varnothing$.

    $\Longleftarrow$: Denote by $\varphi: M \rightarrow \mathbb{P}^{n}$ the continuous map defines $\mathscr{L}$, then by assumption $\exists\ x_{0} \in \Omega$, s.t. $\varphi\left(x_{0}, 0\right) \pitchfork \mathbb{R}^{n} \times\{0\}$. By continuity, up to rigid motion and homothety we assume $\mathbb{D}^{n} \subseteq \Omega$ and $\forall\ x \in \mathbb{D}^{n}, \varphi(x, 0) \pitchfork \mathbb{R}^{n} \times\{0\}$. Take $f \equiv 0$ then $P_{f}=I d_{\mathbb{D}^{n} \times \left\{ 0\right\}}$ and hence $P_{f}^{*} \mathscr{L}=\mathscr{L}$. By Corollary \ref{coro5_7} we have $\operatorname{int}\left(\Delta_{*} \mathscr{L}\right)=\operatorname{int}\left(\Delta_{*} P_{f}^{*} \mathscr{L}\right) \neq \varnothing$.
\end{proof}

We also have the following application to convex functions and convex curves.

\begin{theorem}\label{theo5_9}
    Let $f: I \rightarrow \mathbb{R}$ be a strictly convex function then $\forall$ line bundle $\mathscr{L}$ over $\Gamma_{f}$ we have $\operatorname{int}\left(\Delta_{*} \mathscr{L}\right) \neq \varnothing$.
\end{theorem}
\begin{proof}
    Let $\mathscr{E}$ be the line bundle over $I$ s.t. $P_{f}^{*} \mathscr{E}=\mathscr{L}$ and denote by $\varphi: I \rightarrow \mathbb{P}^{1}$ the continuous map defines $\mathscr{E}$.

    We consider the following two cases:

    Case 1: $\forall\ x \in I, \varphi(x)$ is a supporting hyperplane of the epigraph of $f$ at $(x, f(x))$. For any $a \in I$, since $\Gamma_{f} \cap\left\{(x, y) \in \mathbb{R}^{2} \mid x>a\right\} \neq \varnothing$ and $\Gamma_{f} \cap\left\{(x, y) \in \mathbb{R}^{2} \mid x<a\right\} \neq \varnothing$, we have $\varphi(a)$ is not the vertical line. Denote by $U_{0}:=\{[x: y] \mid x \neq 0\} \subseteq \mathbb{P}^{1}$ the standard affine chart and $\pi_{0}:=[x:y] \mapsto \dfrac{y}{x}$ the chart map, then $\varphi(I) \subseteq U_{0}$. Define $g: I \rightarrow \mathbb{R}$ by $g:=\pi_{0} \circ \varphi$, and denote by $C$ the set of non-differentiable points of $f$. Since $f$ is convex, we have that $C$ is countable and $I\backslash C$ is dense in $I$. By the continuity of $g, g(I \backslash C)$ is dense in $g(I)$ and $g(I)$ is a connected subset of $\mathbb{R}$, i.e. an (finite or infinite) interval.

    Denote by $f_{L}^{\prime}$ the left derivative of $f$, then $C$ is the set of discontinuity of $f_{L}^{\prime}$ and $\forall\ x \in I \backslash C$ we have $f_{L}^{\prime}(x)=f^{\prime}(x)$. Recall that for any $x \in I \backslash C$ the supporting hyperplane of the epigraph of $f$ at $(x, f(x))$ is unique and coincides with the tangent space of $\Gamma_{f}$ at $(x, f(x))$, we obtain that $\left.f_{L}^{\prime}\right|_{I \backslash C} \equiv g\mid_{I \backslash C}$. Therefore $f_{L}^{\prime}(I \backslash C)=$ is dense in interval $g(I)$, and hence $f_{L}^{\prime}$ has no jump discontinuity. Since the one-sided derivative of a convex function is monotonous and hence admits only jump discontinuity, we conclude that $f_{L}^{\prime}$ is continuous. Therefore $C=\varnothing, f_{L}^{\prime} \equiv f^{\prime}$, and $f$ is $C^{1}$. Moreover, since $f$ is strictly convex, we have that $f$ is nonlinear. By Proposition \ref{prop3_3} we have $\operatorname{int}\left(\Delta_{*} T \Gamma_{f}\right) \neq \varnothing$.

    By the differentiability of $f, \forall\ x \in I$ we have $\varphi(x)=T_{(x, f(x))} \Gamma_{f}$. Therefore $\mathscr{L}=T \Gamma_{f}$, and hence $\operatorname{int}\left(\Delta_{*} \mathscr{L}\right) \neq \varnothing$.

    Case 2: $\exists\ x_{0} \in I$ s.t. $\varphi\left(x_{0}\right)$ is not a supporting hyperplane of the epigraph of $f$ at $\left(x_{0}, f\left(x_{0}\right)\right)$.

    Up to rigid motion and homothety of $\mathbb{R}^{2} \supseteq \Gamma_{f}$ wlog we assume that $f\left(x_{0}\right)=0$ is the global minima of $f$. Since the horizontal line $\mathbb{R} \times\{0\}$ is a supporting hyperplane of the epigraph of $f$ at $\left(x_{0}, 0\right)$, we have that $\varphi(0) \neq[1: 0]$. By the continuity of $\varphi$ wlog we assume that $\forall\ x \in I$, $\varphi(x) \neq[1: 0]$, i.e. $\mathscr{E} \pitchfork \mathbb{R} \times\{0\}$.

    Since $f$ is strictly convex and attains global minima at $x_{0} \in I$, there exists $\varepsilon>0$, s.t. $\# f^{-1}(\{\varepsilon\})=2$. Wlog assume that $f^{-1}(\{\varepsilon\})=\{\pm 1\}=\partial \mathbb{D}$ where $\mathbb{D}=[-1,1] \stackrel{j}{\subseteq} I$, then $j^{*} \mathscr{E}$ is a line bundle over $\mathbb{D}$ with $\left.f\right|_{\partial \mathbb{D}} \equiv \varepsilon$ and $j^{*} \mathscr{E} \pitchfork \mathbb{R} \times\{0\}$. Therefore by Corollary \ref{coro5_7} $\operatorname{int}\left(\Delta_{*} P_{\left.f\right|_{\mathbb{D}}}^{*} j^{*} \mathscr{E}\right) \neq \varnothing$. Since $\mathbb{D} \stackrel{j}{\subseteq} I$ induces the inclusion $\Gamma_{\left.f\right|_{\mathbb{D}}} \stackrel{i}{\subseteq} \Gamma_{f}$ where $i=P_{f}^{-1} \circ j \circ P_{\left.f\right|_{\mathbb{D}}}$, we conclude that $\mathscr{L}=P_{f}^{*} \mathscr{E} \supseteq i^{*} P_{f}^{*} \mathscr{E}=P_{\left.f\right|_{\mathbb{D}}}^{*} j^{*} \mathscr{E}$. Therefore $\Delta_{*} \mathscr{L} \supseteq \Delta_{*} P_{\left.f\right|_{\mathbb{D}}}^{*} j^{*} \mathscr{E}$ has non-empty interior.
\end{proof}

\begin{theorem}\label{theo5_10}
    Let $\gamma$ be a strictly convex curve then $\forall$ line bundle $\mathscr{L}$ over $\gamma$ we have $\operatorname{int}\left(\Delta_{*} \mathscr{L}\right) \neq \varnothing$.
\end{theorem}
\begin{proof}
    Decompose $\gamma$ into upper and lower hulls then up to rigid motion and homothety $\exists$ strictly convex function $f: I \rightarrow \mathbb{R}$ s.t. the lower hull contains $\Gamma_{f}\stackrel{i}{\subseteq}\gamma$. By Theorem \ref{theo5_9} int $\left(\Delta_{*} i^{*} \mathscr{L}\right) \neq \varnothing$. Since $\mathscr{L} \supseteq i^{*} \mathscr{L}$ we have $\operatorname{int}\left(\Delta_{*} \mathscr{L}\right) \subseteq \operatorname{int}\left(\Delta_{*} i^{*} \mathscr{L}\right) \neq \varnothing$.
\end{proof}

Notice that the strictly convexity condition in the two theorems above cannot be weakened, since picewise-linear curves clearly do not enjoy the desired property.

The following corollaries show that in the two theorems above if instead of a full line bundle we take any arbitrary $1$ Hausdorff dimensional subset of each fiber then the image under the codiagonal morphism still attains full Hausdorff dimension.

\begin{corollary}\label{coro_5_2_4}
    Let $f: I \longrightarrow \mathbb{R}$ be a strictly convex function, $\pi: \mathscr{L} \rightarrow \Gamma_{f}$ be a line bundle and $\mathscr{L}_{0} \stackrel{i}{\subseteq} \mathscr{L}$ be any subset. If $\forall\ p \in \Gamma_{f}$ we have $\operatorname{dim}_{\mathcal{H}}(\pi \circ i)^{-1}(\{p\})=1$, then $\operatorname{dim}_{\mathcal{H}} \Delta_{*} \mathscr{L}_{0}=2$.
\end{corollary}
\begin{proof}
    By Theorem \ref{theo5_10} in particular we have $\operatorname{dim}_{\mathcal{H}} \Delta_{*} \mathscr{L}=2$.

    Take any $q \in \Delta_{*} \mathscr{L}$ then clearly $\exists\ p \in \Gamma_{f}$ s.t. $q \in \ell:=p+\pi^{-1}(\{p\})$. Since $\ell \cap \Delta_{*} \mathscr{L}_{0} \supseteq\left(p+\pi^{-1}(\{p\})\right) \cap \left(p+(\pi \circ i)^{-1}(\{p\})\right)=p+(\pi \circ i)^{-1}(\{p\})$, we obtain that $\operatorname{dim}_{\mathcal{H}}\left(\ell \cap \Delta_{*} \mathscr{L}_{0}\right) \geqslant \operatorname{dim}_{\mathcal{H}}\left(p+(\pi \circ i)^{-1}(\{p\})\right)=\operatorname{dim}_{\mathcal{H}}(\pi \circ i)^{-1}(\{p\})=1$. Therefore by Corollary 2.5 in \cite{9} we have $\operatorname{dim}_{\mathcal{H}} \Delta_{*} \mathscr{L}_{0} \geqslant 2 \cdot 1-1+\min\left\{1,\operatorname{dim}_{\mathcal{H}} \Delta_{*} \mathscr{L}-1\right\}=1+\min \{1,1\}=2$. Since $2 \leqslant \operatorname{dim}_{\mathcal{H}} \Delta_{*} \mathscr{L}_{0} \leqslant \operatorname{dim}_{\mathcal{H}} \mathbb{R}^{2}=2$, we conclude that $\operatorname{dim}_{\mathcal{H}} \Delta_{*} \mathscr{L}_{0}=2$.
\end{proof}

\begin{corollary}\label{coro_5_2_5}
    Let $\pi: \mathscr{L} \rightarrow \gamma$ be a line bundle over a strictly convex curve, and $\mathscr{L}_{0} \stackrel{i}{\subseteq} \mathscr{L}$ be any subset. If $\forall\ p \in \gamma$ we have $\operatorname{dim}_{\mathcal{H}}(\pi \circ i)^{-1}(\{p\})=1$, then $\operatorname{dim}_{\mathcal{H}}\Delta_{*}\mathscr{L}_0=2$.
\end{corollary}
\begin{proof}
    The same as Corollary \ref{coro_5_2_4}.
\end{proof}

Finally we discuss the application to spheres and ellipsoids.

We shall first show that similar to the hyperplane case, any non tangent line bundle attains its expected size.

\begin{proposition}\label{prop5_11}
    Let $\mathscr{L}$ be a line bundle over an open subset $U$ of a $n$-sphere in $\mathbb{R}^{n+1}$, if $\mathscr{L}$ is not tangent to $U$ then $\operatorname{int}\left(\Delta_{*} \mathscr{L}\right) \neq \varnothing$.
\end{proposition}
\begin{proof}
    Denote by $S \subseteq \mathbb{R}^{n+1}$ the $n$-sphere, wlog we assume that $S$ is centred at 0 with radius $R>1$, then at the vicinity of the north pole $p:=(0, R)$, $S$ is the graph of $f(x)=\sqrt{R^{2}-\|x\|^{2}},\|x\|<R$. Wlog assume that $U$ is an open neighborhood of $p$ and $\mathscr{L}$ is not tangent to $U$ at $p$.

    Denote by $\varphi: U \rightarrow \mathbb{P}^{n}$ the continuous map defines $\mathscr{L}$. Since $T_{p} S=\mathbb{R}^{n} \times\{0\}$ we have $\varphi(0, R) \pitchfork \mathbb{R}^{n} \times\{0\}$. By continuity, up to a homothety $\forall\ x \in \mathbb{D}^{n}$ we have $(x, f(x)) \in U$ and $\varphi(x, f(x)) \pitchfork \mathbb{R}^{n} \times\{0\}$. It is clear that $\left.f\right|_{\mathbb{D}^{n}} \in C\left(\mathbb{D}^{n}\right)$ and $\left.f\right|_{\partial \mathbb{D}^{n}}\equiv \text{Constant}$. Since $\mathscr{L} \subseteq i^{*} \mathscr{L}$ we conclude that $\operatorname{int}\left(\Delta_{*} \mathscr{L}\right) \supseteq\operatorname{int}\left(\Delta_{*} i^{*} \mathscr{L}\right) \neq \varnothing$.
\end{proof}

As an application of Poincaré-Hopf index formula, we know that there exists no continuous nonsingular tangent vector field over even dimensional spheres. Taking this into account, from Proposition \ref{prop5_11} one deduces that any line bundle over an even dimensional sphere generated by a continuous unit vector field attains its expected size.

Actually more is true, as we have the following proposition:
\begin{proposition}\label{prop5_12}
    Let $n \geqslant 1$ and $\mathscr{D}$ be a continuous line distríbution over a $n$-sphere $S$ in $\mathbb{R}^{n+1}$, then $\Delta_{*} \mathscr{D}=\mathbb{R}^{n+1} \backslash B$ where $B$ is the open ball bounded by $S$.
\end{proposition}
\begin{proof}
    If $n=1$ then $S$ is a circle in $\mathbb{R}^{2}$ and $\mathscr{D}=TS$, by observation the statement clearly holds. It remains to proof the statement for $n \geqslant 2$.
    
    Wlog assume that $S:=\mathbb{S}^{n}=\left\{x \in \mathbb{R}^{n} \mid\|x\|=1\right\}\subseteq \mathbb{R}^{n}$, then $B=\{\|x\|<1\}$. 
    
    Since for any $k \geqslant 2$, $H^{1}\left(\mathbb{S}^{k}, \mathbb{Z} / 2 \mathbb{Z}\right)=H^{1}\left(\mathbb{S}^{k}, \mathbb{Z}\right) \otimes \mathbb{Z} / 2 \mathbb{Z}=0 \otimes \mathbb{Z} / 2 \mathbb{Z}=0$, we have that the Stiefel-Whitney class $w_{1}(\mathscr{D})=0$, and hence $\mathscr{D}$ is generated by its non-singular section, i.e. a continuous unit vector field $X: \mathbb{S}^{n} \rightarrow \mathbb{S}^{n}$.

    For any $u \in \mathbb{S}^{n}$ define $\ell_{u}(t):=u+t X(u), t \in \mathbb{R}$, and $\gamma_{u}(t):=u+t X(u), t \in \mathbb{R}_+$, then $\Delta_{*} \mathscr{L}=\bigcup\limits_{u \in \mathbb{S}^{n}} \ell_{u} \supseteq \bigcup\limits_{u \in \mathbb{S}^{n}} \gamma_{u}$. 
    
    Take any $\rho>1$, then by Jordan's curve theorem $\forall\ u \in \mathbb{S}^{n}$, $\gamma_{u} \pitchfork \partial B_{\rho}^{n+1}$ at an unique point $p_{u, \rho}$. Define $\psi_{\rho}: \mathbb{S}^{n} \rightarrow \mathbb{S}^{n}$ by $\psi_{\rho}(u):=\dfrac{p_{u, \rho}}{\| p_{u, \rho \|}}=\dfrac{1}{\rho} p_{u, \rho}$, then it is easy to check that $\psi_{\rho}$ is continuous.

    Take any $u \in \mathbb{S}^{n}$. Since $X(u) \in T_{u} \mathbb{S}^{n}$ we have that $X(u) \perp u$. Therefore for any $t \in \mathbb{R}_{+}$we have $\langle u+t X(u),-u\rangle=-\langle u, u\rangle-t\langle X(u), u\rangle=-1-t \cdot 0=-1$ and hence $\left\langle\dfrac{u+t X(u)}{\|u+t X(u)\|},-u\right\rangle=-\dfrac{1}{\|u+t X(u)\|}<0$, in particular $\dfrac{u+t X(u)}{\|u+t X(u)\|} \neq-u$. Therefore $\forall\ u \in \mathbb{S}^{n}, \forall\ \rho>1$, we have $\psi_{\rho}(u) \neq-u$.

    Take any $\rho>1$. Since $\forall\ u \in \mathbb{S}^{n}$, $\psi_{\rho}(u) \neq-u$ we have that $\psi_{\rho} \simeq I d_{\mathbb{S}^n}$. Since by Brouwer's theorem $I d_{\mathbb{S}^n}$ is not null-homotopic and recall that any non surjective continuous map to a sphere is null-homotopic, we conclude that $\psi_{\rho}$ is surjective. Therefore $\Delta_{*} \mathscr{L} \supseteq \bigcup\limits_{u \in \mathbb{S}^{n}} \gamma_{u} \supseteq\left\{p_{u, \rho}\right\}_{u \in \mathbb{S}^{n}}=\partial B_{\rho}^{n}$.

    Therefore $\Delta_{*} \mathscr{L} \supseteq\bigcup\limits_{\rho>1}\partial B_{\rho}^{n}=\{\|x\|>1\}$.

    Since clearly $\mathbb{S}^{n} \subseteq \Delta_{*} \mathscr{L}$, we have that $\Delta_{*} \mathscr{L} \supseteq\{\|x\| \geqslant 1\}$. Since $\forall\ u \in \mathbb{S}^{n}, \ell_{u}$ is tangent to $\mathbb{S}^{n}$ and hence $\inf\limits_{p \in \ell_{u}}\|p\|=1$, we have that $\forall\ x \in\{\|x\|<1\} x \notin \ell_{u}$. Therefore $\forall\ x \in\{\|x\|<1\}$, we have $x\notin\bigcup\limits_{u\in\mathbb{S}^n}\ell_u=\Delta_{*} \mathscr{L}$.

    This proves that $\Delta_{*} \mathscr{L}=\{\|x\| \geqslant 1\}=\mathbb{R}^{n} \backslash\{\|x\|<1\}$.
\end{proof}

Now we state our conclusive result for spheres.

\begin{theorem}\label{theo5_13}
    Let $\mathscr{L}$ be a line bundle over a $n$-sphere $S$ in $\mathbb{R}^{n+1}$ then $\operatorname{int}\left(\Delta_{*} \mathscr{L}\right) \neq \varnothing$.
\end{theorem}
\begin{proof}
    If $n=0$ then the ambient space is $\mathbb{R}^{1}$ and clearly $\Delta_{*} \mathscr{L}=\mathbb{R}^{1}$. In particular we have that $\operatorname{int}\left(\Delta_{*} \mathscr{L}\right) \neq \varnothing$.

    If $n \geqslant 1$ and $\mathscr{L}$ is not tangent to $S$ then by Proposition \ref{prop5_11} int $\left(\Delta_{*} \mathscr{L}\right) \neq \varnothing$. 
    
    If $n \geqslant 1$ and $\mathscr{L}$ is tangent to $S$ then by Proposition \ref{prop5_12} $\operatorname{int}\left(\Delta_{*} \mathscr{L}\right) \neq \varnothing$.
\end{proof}

\begin{remark}\label{remark5_14}
    By definition, ellipsoids are regular level sets of positive determined quadratic forms. Since by an affine change of coordinates one can reduce a quadratic form to its canonical form, equivalently we have that ellipsoids are images of the unit sphere under non-degenerated affine transforms. Therefore the statements in Proposition \ref{prop5_11}, \ref{prop5_12} and Theorem \ref{theo5_13} hold also for ellipsoids. In particular, any line bundle over an ellipsoid attains its expected size.
\end{remark}
\begin{remark}\label{remark5_15}
    If the family of lines is not parameterized continuously, then the assertions in this section do not hold in general. For example, see \cite{5} for detailed construction of counter examples for convex curves and high dimensional generalizations.
\end{remark}

\section{Acknowledgements}
I would like to express my sincere gratitude to Professor Keleti who taught me geometric measure theory. I would like to express my special thanks to Professor Keleti and Professor Csíkós for very useful comments and suggestions.

\printbibliography

\quad

\address
{\scshape{University of Victoria
3800 Finnerty Road 
Victoria BC  V8P 5C2 
Canada}}

{\textit{Email address:}} \email{hanwenliu0804@outlook.com}

\end{document}